\documentclass[10pt,a4paper,reqno]{amsart}
\usepackage{amsthm}
\usepackage{amsmath}
\usepackage{amssymb}
\usepackage{graphicx,color}

\usepackage[font=small]{caption}

\usepackage{multirow}
\usepackage[shortlabels]{enumitem}

\makeatletter
\theoremstyle{plain}
\newtheorem{thm}{Theorem}
  \theoremstyle{definition}
  
  \theoremstyle{remark}
  \newtheorem{rem}[thm]{Remark}
  \theoremstyle{plain}
  
  \theoremstyle{plain}
  \newtheorem{lem}[thm]{Lemma}
  \theoremstyle{plain}
  \newtheorem{cor}[thm]{Corollary}
 \theoremstyle{definition}
  \newtheorem{example}[thm]{Example}
  \theoremstyle{remark}
  \newtheorem*{rem*}{Remark}

  \theoremstyle{definition}

\usepackage{amsfonts}
\usepackage{mathrsfs}

\newtheorem*{question*}{\it{QUESTION}}

\theoremstyle{plain}
\newtheorem{conjecture}{Conjecture}

\newcommand{\N}{\mathbb{N}}
\newcommand{\R}{{\mathbb{R}}}
\newcommand{\C}{{\mathbb{C}}}
\newcommand{\Z}{{\mathbb{Z}}}

\newcommand{\dd}{{\rm d}}
\newcommand{\ii}{{\rm i}}

\renewcommand{\Re}{\mathop\mathrm{Re}\nolimits}
\renewcommand{\Im}{\mathop\mathrm{Im}\nolimits}

\DeclareMathOperator{\arccot}{arccot}

\makeatother

\begin{document}

\title[]{Asymptotic behavior and zeros of the Bernoulli polynomials of the second kind}

\author{Franti{\v s}ek {\v S}tampach}
\address[Franti{\v s}ek {\v S}tampach]{
	Department of Applied Mathematics, Faculty of Information Technology, Czech Technical University in~Prague, 
	Th{\' a}kurova~9, 160~00 Praha, Czech Republic
	}
\email{stampfra@fit.cvut.cz}



\date{\today}

\begin{abstract}
The main aim of this article is a careful investigation of the asymptotic behavior of zeros of Bernoulli polynomials of the second kind. It is shown that the zeros are all real and simple. The asymptotic expansions for the small, large, and the middle zeros are computed in more detail. The analysis is based on the asymptotic expansions of the Bernoulli polynomials of the second kind in various regimes.
\end{abstract}

\maketitle

\section{Introduction}

The Bernoulli polynomials of the second kind $b_{n}$ are defined by the generating function
\begin{equation}
 \sum_{n=0}^{\infty}b_{n}(x)\frac{t^{n}}{n!}=\frac{t}{\ln(1+t)}(1+t)^{x}, \quad |t|<1.
\label{eq:gener_func_BP_sec}
\end{equation}
Up to a shift, they coincide with the generalized Bernoulli polynomials of order $n$, precisely, 
\begin{equation}
 b_{n}(x)=B_{n}^{(n)}(x+1),
 \label{eq:BP_sec_rel_B}
\end{equation}
where the generalized Bernoulli polynomials of order $a\in\C$ are defined by the generating function
\begin{equation}
  \sum_{n=0}^{\infty}B_{n}^{(a)}(x)\frac{t^{n}}{n!}=\left(\frac{t}{e^{t}-1}\right)^{a}e^{xt}, \quad |t|<2\pi.
\label{eq:gener_B_n_a}
\end{equation}
Among the numerous branches of mathematics where the generalized Bernoulli polynomials play a significant role, we emphasize the theory of finite integration and interpolation which is nicely exposed in the classical treatise of N\"{o}rlund~\cite{norlund_24} (unfortunately still not translated to English).

In contrast to the Bernoulli polynomials of the first kind~$B_{n}\equiv B_{n}^{(1)}$, the second kind Bernoulli polynomials appear less frequently. Still a great account of the research focuses on the study of their various generalizations and their combinatorial, algebraic, and analytical properties; let us mention at least~\cite{carlitz-sm61,guo_etal-rmjm16,gupta_prabhakar-ijpam80,kim_etal-ade14,roman84,srivastava-choi01, srivastava-todorov_jmaa88}. Concerning the Bernoulli polynomials of the first kind~$B_{n}$, a significant part of the research is devoted to the study of the asymptotic properties of their zeros that exhibit a fascinating complexity and structure; the asymptotic zero distribution of~$B_{n}$, for $n\to\infty$, was obtained in~\cite{boyer-goh_aam07}, for the results on the asymptotic behavior of real zeros of~$B_{n}$ for $n$ large, see~\cite{delange_86,delange_91,efimov_fm08,inkeri-auts59,veselov-ward_jmaa05}, the proof that the zeros of~$B_{n}$ are actually all simple is in~\cite{brillhart-jram69,dilcher_aa08}, and further results on the number of real zeros of~$B_{n}$ can be found in~\cite{edwards-leeming_jat12,leeming_jat89}.

On the other hand, it seems that the asymptotic behavior of zeros of the Bernoulli polynomials of the second kind has not been studied yet. This contrast was a motivation for the present work whose main goal is to fill this gap. To this end, let us mention that the asymptotic behavior of the polynomials $B_{n}^{(a)}$, for $n$ large and the order~$a$ not depending on~$n$, was studied in~\cite{lopez-temme_jmaa99,lopez-temme_jmaa10}.

Throughout the paper, we prefer to work with the polynomials $B_{n}^{(n)}$ rather than the Bernoulli polynomials of the second kind~$b_{n}$ themselves which, however, makes no difference due to the simple relation~\eqref{eq:BP_sec_rel_B}. First, we prove that, unlike zeros of the Bernoulli polynomials of the first kind, the zeros of the second kind Bernoulli polynomials can be only real. Moreover, all zeros are simple, located in the interval~$[0,n]$, and interlace with the integers $1,2,\dots,n$ (Theorem~\ref{thm:Ber_sec_zer_real}).

Next, we focus on the small zeros of~$B_{n}^{(n)}$, i.e., the zeros that are located in a fixed distance from the origin. The proof of their asymptotic behavior (Theorem~\ref{thm:asympt_zer_small}) is based on a complete local uniform asymptotic formula for~$B_{n}^{(n)}$ (Theorem~\ref{thm:Ber_sec_compl_asympt}). It turns out that the zeros of~$B_{n}^{(n)}$ are distributed symmetrically around the point $n/2$ and hence we obtain also the asymptotic formulas for the large zeros, i.e., the zeros located in a fixed distance from $n$.

Further, the asymptotic behavior of zeros of~$B_{n}^{(n)}$ located in a fixed distance from a point~$\alpha n$, where $\alpha\in(0,1)$, is examined. The analysis uses an interesting integral representation of $B_{n}^{(n)}$ which can be a formula of independent interest (Theorem~\ref{thm:beta_integr_repre}). With the aid of Laplace's method, the leading term of the asymptotic expansion of~$B_{n}^{(n)}(z+\alpha n)$, as $n\to\infty$, is deduced (Theorem~\ref{thm:asympt_ber_alpha}) and limit formulas for the zeros located around $\alpha n$ then follow (Corollary~\ref{cor:lim_zer_alp}). A particular attention is paid to the middle zeros, i.e., the case $\alpha=1/2$. In this case, more detailed results are obtained. First, a complete local uniform asymptotic expansion for $B_{n}^{(n)}(z+n/2)$, as $n\to\infty$, is derived (Theorem~\ref{thm:asympt_ber_alpha1/2}). As a consequence, we obtain several terms in the asymptotic expansion of the middle zeros (Theorem~\ref{thm:asympt_middle_zer}).

The asymptotic formulas for~$B_{n}^{(n)}$ that are used to analyze the zeros can be viewed as the asymptotic expansion of the scaled polynomials $B_{n}^{(n)}(nz)$ in the oscilatory region $z\in(0,1)$ and close to the edge points $z=0$ and $z=1$. To complete the picture, the leading term of the asymptotic expansion of $B_{n}^{(n)}(nz)$ in the non-oscilatory regime or the zero-free region $z\in\C\setminus[0,1]$ is derived (Theorem~\ref{thm:asympt_ber_non-osc}) using the saddle point method. Finally, we formulate several open research problems in the end of the paper.

\section{Asymptotic behavior and zeros of the Bernoulli polynomials of the second kind}

First, we recall several basic properties of the polynomials $B_{n}^{(a)}$ that are used in this section. For $n\in\N$ and $x,a\in\C$, the identities 
\begin{equation}
 B_{n}^{(a)}(x+1)=B_{n}^{(a)}(x)+nB_{n-1}^{(a-1)}(x), \quad \frac{\partial}{\partial x}B_{n}^{(a)}(x)=nB_{n-1}^{(a)}(x)
\label{eq:ber_a_ident1}
\end{equation}
and
\begin{equation}
 B_{n}^{(a)}(-x)=(-1)^{n}B_{n}^{(a)}(x+a)
 \label{eq:symmetry_id_genBP}
\end{equation}
can be derived readily from~\eqref{eq:gener_B_n_a}, see also~\cite[Chp.~4, Sec.~2]{roman84}. Next, by making use of the special value 
\begin{equation}
 B_{n-1}^{(n)}(x)=\prod_{k=1}^{n-1}(x-k)
\label{eq:B_n-1_n_explicit}
\end{equation}
together with the identities from~\eqref{eq:ber_a_ident1}, one easily deduces the well-known integral formula
\begin{equation}
 B_{n}^{(n)}(x)=\int_{0}^{1}\prod_{k=1}^{n}(x+y-k)\,\dd y.
 \label{eq:int_repre_simple}
\end{equation}

\subsection{Real and simple zeros}

An elementary proof of the reality and simplicity of the zeros of the Bernoulli polynomials of the second kind is based on an inspection of the integer values of these polynomials.

\begin{thm}\label{thm:Ber_sec_zer_real}
 Zeros of $B_{n}^{(n)}$ are real and simple. In addition, if $x_{1}^{(n)}<x_{2}^{(n)}<\dots<x_{n}^{(n)}$ denotes the zeros of $B_{n}^{(n)}$,
 then
 \begin{equation}
 k-1<x_{k}^{(n)}<k,
 \label{eq:loc_zeros_first}
 \end{equation}
 for $1\leq k\leq n$. Further, the zeros of $B_{n}^{(n)}$ are distributed symmetrically around the value $n/2$, i.e.,
 \begin{equation}
  x_{k}^{(n)}=n-x_{n-k+1}^{(n)},
 \label{eq:symmetry_zeros}
 \end{equation}
 for $1\leq k\leq n$. In particular, $x_{n}^{(2n-1)}=n-1/2$.
\end{thm}

\begin{proof}
 We start with the integral representation~\eqref{eq:int_repre_simple} which implies that
 \[
  B_{n}^{(n)}(k)=\int_{0}^{1}\prod_{j=1}^{n}\left(y+k-j\right)\dd y=
  (-1)^{n+k}\int_{0}^{1}\left[\prod_{i=0}^{k-1}\left(y+i\right)\right]\left[\prod_{j=1}^{n-k}\left(j-y\right)\right]\dd y,
 \]
 for any integer $0\leq k\leq n$. Since each factor in the above integral is positive for $y\in(0,1)$, the whole integral has to be positive and therefore
 \begin{equation}
  (-1)^{n+k}B_{n}^{(n)}(k)>0,
 \label{eq:int_val_sign}
 \end{equation}
 for $0\leq k\leq n$. 
 
 Consequently, the signs of the values $B_{n}^{(n)}(k)$ alternate for $0\leq k\leq n$ and hence there has to be at least one root of $B_{n}^{(n)}$ in each interval $(k-1,k)$ for $1\leq k\leq n$. Since the polynomial $B_{n}^{(n)}$ is of
 degree $n$, there has to be exactly one zero in each interval $(k-1,k)$, for $1\leq k\leq n$. Thus, $B_{n}^{(n)}$ has $n$ distinct roots
 located in the intervals $(k-1,k)$, $1\leq k\leq n$. These roots are necessarily simple.
 
 The symmetry of the distribution of the roots around $n/2$ follows readily from the identity~\eqref{eq:symmetry_id_genBP} which implies
 \begin{equation}
   B_{n}^{(n)}\left(\frac{n}{2}-x\right)=(-1)^{n}B_{n}^{(n)}\left(\frac{n}{2}+x\right)\!,
 \label{eq:symmetry_ber_sec}
 \end{equation}
 for $n\in\N_{0}$.
\end{proof}

 According to the Gauss--Lucas theorem, the zeros of 
 \[ 
  B_{k}^{(n)}(x)=\frac{k!}{n!}\frac{\dd^{n-k}}{\dd x^{n-k}}B_{n}^{(n)}(x),
 \]
 where $0\leq k\leq n$, are located in the convex hull of the zeros of $B_{n}^{(n)}$ which is, by Theorem~\ref{thm:Ber_sec_zer_real}, a subset of the interval~$(0,n)$.
 
\begin{cor}
 For all $n\in\N_{0}$ and $0\leq k\leq n$, the zeros of $B_{k}^{(n)}$ are located in $(0,n)$.
\end{cor}

\subsection{The asymptotic expansion of $B_{n}^{(n)}(z)$ and small and large zeros}

First, we derive a complete locally uniform asymptotic expansion of~$B_{n}^{(n)}$ in negative powers of $\log n$. As an application, this expansion allows us to derive asymptotic formulas for the zeros of~$B_{n}^{(n)}$ that are located in a fixed distance from the origin or the point~$n$, for $n$ large.

In the proof of the asymptotic expansion of~$B_{n}^{(n)}$, we will make use of a particular case of Watson's lemma given below; for the more general version and its proof, see, e.g., \cite[Chp.~3, Thm.~3.1]{olver_97} or \cite[Sec.~I.5]{wong01}.

\begin{lem}[Watson]\label{lem:Watson}
 Let $f(u)$ be a function of positive variable $u$, such that
 \begin{equation}
  f(u)\sim\sum_{m=0}^{\infty}a_{m}u^{m+\lambda-1}, \quad \mbox{ as }u\to0+,
 \label{eq:f_asympt_watson}
 \end{equation}
 where $\lambda>0$. Then one has the complete asymptotic expansion
 \begin{equation}
  \int_{0}^{\infty}e^{-xu}f(u)\dd u\sim\sum_{m=0}^{\infty}\Gamma\!\left(m+\lambda\right)\frac{a_{m}}{x^{m+\lambda}}, \quad \mbox{ as }x\to\infty,
 \label{eq:int_asympt_watson}
 \end{equation}
 provided that the integral converges absolutely for all sufficiently large $x$.
\end{lem}

\begin{rem}\label{rem:uniform_watson}
 If additionally the coefficients $a_{m}=a_{m}(\xi)$ in~\eqref{eq:f_asympt_watson} depend continuously on a parameter $\xi\in K$ where $K$ is a compact subset of $\C$ and the asymptotic expansion~\eqref{eq:f_asympt_watson} is uniform in $\xi\in K$, then the expansion~\eqref{eq:int_asympt_watson} holds uniformly in $\xi\in K$ provided that the integral converges uniformly in $\xi\in K$ for all sufficiently large~$x$. This variant of Watson's lemma can be easily verified by a slight modification of the proof given, for example, in~\cite[Chp.~3, Thm.~3.1]{olver_97}.
\end{rem}

Yet another auxiliary statement, this time an inequality for the Gamma function with a complex argument, will be needed to obtain the desired asymptotic expansion of $B_{n}^{(n)}(z)$ for $n\to\infty$.

\begin{lem}\label{lem:gamma_ineq}
 For all $s\in[0,1]$ and $z\in\C$ such that $\Re z>1$, it holds
 \[
  \left|z^{s}\frac{\Gamma(z-s)}{\Gamma(z)}-1\right|\leq\frac{2s}{\Re z -1}\left|\frac{z+1}{z-s}\right|.
 \]
\end{lem}

\begin{proof}
 Let $f_{z}(s):=z^{s}\Gamma(z-s)$, $\Re z>1$, and $s\in[0,1]$. Then, by the Lagrange theorem, 
 \begin{equation}
  |f_{z}(s)-f_{z}(0)|\leq|f_{z}'(s^{*})|s,
 \label{eq:dif_f_Lagrange}
 \end{equation}
 for some $s^{*}\in(0,s)$. The differentiation of~$f_{z}(s)$ with respect to~$s$ yields
 \begin{equation}
  f_{z}'(s)=z^{s}\Gamma(z-s)\left(\log z -\psi(z-s)\right)\!,
 \label{eq:der_f_z}
 \end{equation}
 where $\psi=\Gamma'/\Gamma$ is the Digamma function. Recall that \cite[Eq.~5.9.13]{dlmf}
 \[
  \psi(z)=\log z + \int_{0}^{\infty}\left(\frac{1}{t}-\frac{1}{1-e^{-t}}\right)e^{-t z}\dd t,
 \]
 for $\Re z>0$. From the above formula, one deduces that, for $\Re z>1$ and $s\in[0,1]$,
 \[
  |\psi(z-s)-\log(z-s)|\leq\int_{0}^{\infty}e^{-t(\Re z-1)}\dd t=\frac{1}{\Re z-1},
 \]
 where we used that
 \[
  \left|\frac{1}{t}-\frac{1}{1-e^{-t}}\right|<1, \quad \forall t>0.
 \]
 Further, one has
 \[
  |\log(z-s)-\log z|\leq\sum_{n=1}^{\infty}\frac{1}{n}\frac{s^{n}}{|z|^{n}}\leq\frac{1}{|z|-1}\leq\frac{1}{\Re z -1},
 \]
 for $\Re z>1$ and $s\in[0,1]$. Consequently, we get from~\eqref{eq:der_f_z} the estimate
 \[
  |f_{z}'(s)|\leq\frac{2}{\Re z-1}\left|z^{s}\Gamma(z-s)\right|,
 \]
 for all $s\in[0,1]$ and $\Re z>1$.
 
 Recalling~\eqref{eq:dif_f_Lagrange}, the statement is proven, if we show that
 \[
 \left|\frac{z^{s}\,\Gamma(z-s)}{\Gamma(z)}\right|\leq\left|\frac{z+1}{z-s}\right|,
 \]
 for $s\in[0,1]$ and $\Re z>1$. To this end, we apply the inequality~\cite[Eq.~5.6.8]{dlmf}
 \begin{equation}
 \left|\frac{\Gamma(z+a)}{\Gamma(z+b)}\right|\leq\frac{1}{|z|^{b-a}},
 \label{eq:gamma_ratio_ineq}
 \end{equation}
 which holds true provided that $b-a\geq1$, $a\geq0$, and $\Re z>0$. Then
 \[
 \left|\frac{\Gamma(z-s)}{\Gamma(z)}\right|=\left|\frac{z(z+1)}{z-s}\right|\left|
 \frac{\Gamma(z+1-s)}{\Gamma(z+2)}\right|\leq\frac{1}{|z|^{s}}\left|\frac{z+1}{z-s}\right|,
 \]
 where we applied~\eqref{eq:gamma_ratio_ineq} with $a=1-s$ and $b=2$.
\end{proof}

Now, we are at the position to deduce the complete asymptotic expansion of the polynomials $B_{n}^{(n)}(z)$ for $n\to\infty$.

\begin{thm}\label{thm:Ber_sec_compl_asympt}
 The asymptotic expansion
 \begin{equation}
  (-1)^{n}\frac{n^{z}}{n!}B_{n}^{(n)}(z)\sim\sum_{k=0}^{\infty}\frac{c_{k}(z)}{\log^{k+1}n}, \quad n\to\infty,
 \label{eq:asympt_ber_sec}
 \end{equation}
 holds locally uniformly in $z\in\C$, where
 \begin{equation}
  c_{k}(z)=\frac{\dd^{k}}{\dd z^{k}}\frac{1}{\Gamma(1-z)}.
 \label{eq:coeff_c_asympt}
 \end{equation}
\end{thm}

\begin{proof}
 The integral formula~\eqref{eq:int_repre_simple} can be rewritten in terms of the Gamma functions as
 \[
  B_{n}^{(n)}(z)=(-1)^{n}\int_{0}^{1}\frac{\Gamma(1-s-z+n)}{\Gamma(1-s-z)}\dd s.
 \]
 Using this together with Lemma~\ref{lem:gamma_ineq}, we obtain 
 \begin{align*}
  &\left|\frac{B_{n}^{(n)}(z)}{\Gamma(1-z+n)}-(-1)^{n}\int_{0}^{1}\frac{\dd s}{(1-z+n)^{s}\Gamma(1-s-z)}\right|\\
  &\leq\frac{2|n+2-z|}{n-\Re z}\int_{0}^{1}\frac{s\dd s}{|n+1-z-s||(1-z+n)^{s}\Gamma(1-s-z)|},
 \end{align*}
 for $\Re z<n$. Since $1/\Gamma$ is an entire function, for $K\subset\C$ a compact set, there is a constant $C>0$ such that
 \[
  \sup_{s\in[0,1]}\sup_{z\in K}\frac{1}{|\Gamma(1-z-s)|}\leq C.
 \]
 Moreover, one has 
 \[
  \left|\frac{n+2-z}{n+1-z-s}\right|\leq\frac{n+2+|z|}{n-|z|},
 \] 
 for $s\in[0,1]$ and $|z|<n$.
 Hence we have the estimate
 \[
  \left|\frac{B_{n}^{(n)}(z)}{\Gamma(1-z+n)}-(-1)^{n}\int_{0}^{1}\frac{\dd s}{(1-z+n)^{s}\Gamma(1-s-z)}\right|\leq\frac{2C}{n-\Re z}\frac{n+2+|z|}{n-|z|}\int_{0}^{1}\frac{s\dd s}{|1-z+n|^{s}},
 \]
 provided that $|z|<n$. Moreover, since
 \[
 \int_{0}^{1}\frac{s\dd s}{|1-z+n|^{s}}=\frac{|1-z+n|-\log|1-z+n|-1}{|1-z+n|\log^{2}|1-z+n|}\leq\frac{1}{\log^{2}|1-z+n|},
 \]
 we conclude that
 \begin{align}
  \frac{B_{n}^{(n)}(z)}{\Gamma(1-z+n)}&=(-1)^{n}\int_{0}^{1}\frac{\dd s}{(1-z+n)^{s}\Gamma(1-s-z)}+O\left(\frac{1}{n\log^{2}n}\right)\nonumber\\
  &=(-1)^{n}\int_{0}^{1}\frac{\dd s}{n^{s}\Gamma(1-s-z)}+O\left(\frac{1}{n}\right)\!,\label{eq:towards_asympt_exp_inproof0}
 \end{align}
 as $n\to\infty$, locally uniformly in $z\in\C$.
 
 The analyticity of the reciprocal Gamma function implies
 \[
  \frac{1}{\Gamma(1-s-z)}=\sum_{k=0}^{\infty}c_{k}(z)\frac{s^{k}}{k!},
 \]
 where
 \[
  c_{k}(z)=\frac{\dd^{k}}{\dd s^{k}}\bigg|_{s=0}\frac{1}{\Gamma(1-s-z)}=\frac{\dd^{k}}{\dd z^{k}}\frac{1}{\Gamma(1-z)}.
 \]
 Moreover, $c_{k}$ is an entire function for any $k\in\N_{0}$. Consequently, if $\chi_{(0,1)}$ denotes the indicator function of the interval $(0,1)$, then
 \[
  f_{z}(s):=\frac{\chi_{(0,1)}(s)}{\Gamma(1-s-z)}\sim\sum_{k=0}^{\infty}c_{k}(z)\frac{s^{k}}{k!}, \quad s\to0+,
 \]
 and the application of Lemma~\ref{lem:Watson} yields
 \begin{equation}
  \int_{0}^{1}\frac{\dd s}{n^{s}\Gamma(1-s-z)}=\int_{0}^{\infty}e^{-s\log n}f_{z}(s)\dd s\sim\sum_{k=0}^{\infty}\frac{c_{k}(z)}{\log^{k+1}n}, \quad n\to\infty.
 \label{eq:asympt_int_log_watson}
 \end{equation}
 This asymptotic formula is again local uniform in $z\in\C$. To see this, one may proceed as follows. For $m\in\N_{0}$, we have
 \[
  \left|\frac{1}{\Gamma(1-s-z)}-\sum_{k=0}^{m}c_{k}(z)\frac{s^{k}}{k!}\right|\leq\left|\frac{\dd^{m+1}}{\dd s^{m+1}}\bigg|_{s=s^{*}}\frac{1}{\Gamma(1-s-z)}\right|\frac{s^{m+1}}{(m+1)!}=|c_{m+1}(z+s^{*})|\frac{s^{m+1}}{(m+1)!},
 \]
 where $s^{*}\in(0,s)$. Now, if $K\subset\C$ is compact, then, by the analyticity of $c_{m+1}$, there is a constant $C'>0$ such that
 \[
  \sup_{s\in[0,1]}\sup_{z\in K}|c_{m+1}(z+s)|\leq C'.
 \]
 Consequently, the error term in the expansion~\eqref{eq:asympt_int_log_watson} is majorized by
 \[
  \frac{C'}{(m+1)!}\int_{0}^{\infty}\frac{s^{m+1}}{n^{s}}\dd s=\frac{C'}{\log^{m+2} n},
 \]
 for all $z\in K$; cf. also Remark~\ref{rem:uniform_watson}. Thus, \eqref{eq:towards_asympt_exp_inproof0} together with~\eqref{eq:asympt_int_log_watson} imply that
 \begin{equation}
  \frac{B_{n}^{(n)}(z)}{\Gamma(1-z+n)}\sim(-1)^{n}\sum_{k=0}^{\infty}\frac{c_{k}(z)}{\log^{k+1}n}, \quad n\to\infty,
 \label{eq:asympt_ber_sec_inproof}
 \end{equation}
 locally uniformly in $z\in\C$.
 
 Finally, it follows from the well-known Stirling asymptotic expansion that
\[
 \frac{1}{\Gamma(1-z+n)}=\frac{n^{z}}{n!}\left(1+O\left(\frac{1}{n}\right)\right), \quad n\to\infty,
\]
locally uniformly in $z\in\C$. By applying the above formula in~\eqref{eq:asympt_ber_sec_inproof}, we arrive at the statement.
\end{proof}

\begin{rem}
Using different arguments based on a contour integration, a more general result than~\eqref{eq:asympt_ber_sec} was already obtained by N{\" o}rlund in~\cite{norlund_rcmp61} (in French). Nemes derived a complete asymptotic expansion of the Bernoulli numbers of the second kind $B_{n}^{(n)}(1)/n!$ in~\cite{nemes_jis11} which is a particular case of Theorem~\ref{thm:Ber_sec_compl_asympt} for $z=1$. Another asymptotic formula for the Bernoulli numbers of the second kind was deduced by Van Veen in~\cite{vanveen_51}.
\end{rem}

By using Theorem~\ref{thm:Ber_sec_compl_asympt}, we immediately get the following corollary,
where the second statement follows from the Hurwitz theorem, see~\cite[Thm.~2.5, p.~152]{conway_78}, and 
the fact that the zeros of $1/\Gamma$ are located at non-positive integers and are simple.

\begin{cor}\label{cor:hurwitz_small_zer}
 It holds
 \[
  \lim_{n\to\infty}(-1)^{n}\frac{n^{z}\log n}{n!}B_{n}^{(n)}(z)=\frac{1}{\Gamma(1-z)}
 \]
 uniformly in compact subsets of $\C$. Consequently, for $k\in\N$, one has
 \[
  \lim_{n\to\infty}x_{k}^{(n)}=k.
 \]
\end{cor}

\begin{rem}
 Both sequences above converge quite slowly. The error terms turn out to decay as $1/\log n$, for $n\to\infty$.
\end{rem}

Since we known all the coefficients in the asymptotic expansion~\eqref{eq:asympt_ber_sec} by any power of $1/\log n$,
we can compute also the asymptotic expansions of zeros $x_{k}^{(n)}$ for any fixed $k\in\N$, as $n\to\infty$, to an arbitrary order, in principle. However, the coefficients by the powers of $1/\log n$ become quickly complicated and no closed formula for the coefficients was found. We provide the first three terms of the asymptotic expansions.

\begin{thm}\label{thm:asympt_zer_small}
 For $k\in\N$, we have the asymptotic expansion
 \[
  x_{k}^{(n)}=k-\frac{1}{\log n}-\frac{\psi(k)}{\log^{2}n}+O\left(\frac{1}{\log^{3}n}\right)\!, \quad \mbox{ as } n\to\infty,
 \]
 where $\psi=\Gamma'/\Gamma$ is the Digamma function.
\end{thm}

\begin{rem}\label{rem:asympt_zer_large}
 Theorem~\ref{thm:asympt_zer_small} gives the asymptotic behavior of zeros located in a fixed distance from $0$.
 As a consequence of the symmetry~\eqref{eq:symmetry_zeros}, we know also the asymptotic behavior of the zeros
 located in a fixed distance from $n$. Namely, 
 \[
  x_{n-k+1}^{(n)}=n-k+\frac{1}{\log n}+\frac{\psi(k)}{\log^{2}n}+O\left(\frac{1}{\log^{3}n}\right)\!, \quad \mbox{ as } n\to\infty,
 \]
 for any fixed $k\in\N$.
\end{rem}

\begin{proof}
 We fix $k\in\N$ and introduce $\epsilon_{n}^{(k)}:=k-x_{k}^{(n)}$. If we substitute for $z=k-\epsilon_{n}^{(k)}$ in~\eqref{eq:coeff_c_asympt}, we get
 \begin{equation}
 c_{j}\left(k-\epsilon_{n}^{(k)}\right)
 =(-1)^{j}\frac{\dd^{j}}{\dd z^{j}}\bigg|_{z=\epsilon_{n}^{(k)}}\frac{1}{\Gamma(1-k+z)}=(-1)^{j}\frac{\dd^{j}}{\dd z^{j}}\bigg|_{z=\epsilon_{n}^{(k)}}\frac{(z+1-k)_{k}}{\Gamma(z+1)},
 \label{eq:c_j_tow_exp_inproof}
 \end{equation}
 where $(a)_{k}:=a(a+1)\dots(a+k-1)$ is the Pochhammer symbol.
 
 Recall that, for $k\in\N$, 
 \begin{equation}
  (z+1-k)_{k}=\sum_{l=1}^{k}s(k,l)z^{l},
\label{eq:pochham_expand}
 \end{equation}
 where $s(k,l)$ are the Stirling numbers of the first kind; see~\cite[Sec.~26.8]{dlmf}. The Stirling numbers can be defined recursively but here we will make use only of the special values
 \begin{equation}
  s(k,1)=(-1)^{k-1}(k-1)! \quad\mbox{ and }\quad  s(k,2)=(-1)^{k-1}(k-1)!\left(\gamma+\psi(k)\right)\!,
  \label{eq:stirling_spec_val}
 \end{equation}
 for $k\in\N$. We will also need the Maclaurin series for the reciprocal Gamma function~\cite[Eqs.~(5.7.1) and~(5.7.2)]{dlmf}
 \begin{equation}
  \frac{1}{\Gamma(z+1)}=\sum_{m=0}^{\infty}\pi_{m}z^{m},
 \label{eq:recip_gamma_expand}
 \end{equation}
 where 
 \[
  \pi_{0}=1 \quad \mbox{ and } \quad m\pi_{m}=\gamma \pi_{m-1}+\sum_{j=2}^{m}(-1)^{j+1}\zeta(j)\pi_{m-j}, \quad \mbox{ for } m\in\N,
 \]
 and $\zeta$ stands for the Riemann zeta function. In particular,
 \begin{equation}
  \pi_{0}=1 \quad\mbox{ and }\quad \pi_{1}=\gamma.
  \label{eq:recip_gamma_spec_val}
 \end{equation}

 By using~\eqref{eq:pochham_expand} and~\eqref{eq:recip_gamma_expand} in~\eqref{eq:c_j_tow_exp_inproof}, we obtain
 \begin{equation}
  c_{j}\left(k-\epsilon_{n}^{(k)}\right)=(-1)^{j}j!\sum_{m=j}^{\infty}\binom{m}{j}X_{m}^{(k)}\left(\epsilon_{n}^{(k)}\right)^{\!m-j},
 \label{eq:c_j_exp_inproof}
 \end{equation}
 where
 \begin{equation}
  X_{m}^{(k)}:=\sum_{i=1}^{\min(k,m)}s(k,i)\pi_{m-i}.
  \label{eq:def_X_m_k}
 \end{equation}

 It follows from Theorem~\ref{thm:Ber_sec_compl_asympt} that $\epsilon_{n}^{(k)}=O(1/\log n)$. Therefore we may write
 \begin{equation}
  \epsilon_{n}^{(k)}=\frac{\mu_{n}^{(k)}}{\log n},
 \label{eq:def_mu_n_k}
 \end{equation}
 where $\mu_{n}^{(k)}$ is a bounded sequence. If we take the first two terms from the expansion~\eqref{eq:asympt_ber_sec}
 and use that the left-hand side of~\eqref{eq:asympt_ber_sec} vanishes at $z=x_{k}^{(n)}$, we arrive at the equation
 \[
 c_{0}\!\left(k-\epsilon_{n}^{(k)}\right)\log n+c_{1}\!\left(k-\epsilon_{n}^{(k)}\right)+O\left(\frac{1}{\log n}\right)=0, \quad \mbox{ as } n\to\infty.
 \]
 With the aid of~\eqref{eq:c_j_exp_inproof} and~\eqref{eq:def_mu_n_k}, the above equation can be written as
 \[
  X_{1}^{(k)}\mu_{n}^{(k)}-X_{1}^{(k)}+O\left(\frac{1}{\log n}\right)=0, \quad \mbox{ as } n\to\infty,
 \]
 which implies
 \[
  \mu_{n}^{(k)}=1+O\left(\frac{1}{\log n}\right)\!, \quad \mbox{ as } n\to\infty,
 \]
 since $X_{1}^{(k)}\neq0$. Consequently, we get
 \[
  x_{k}^{(n)}=k-\epsilon_{n}^{(k)}=k-\frac{1}{\log n}+O\left(\frac{1}{\log^{2}n}\right)\!, \quad \mbox{ as } n\to\infty.
 \]

 Similarly, by writing
 \begin{equation}
  \epsilon_{n}^{(k)}=\frac{1}{\log n}+\frac{\nu_{n}^{(k)}}{\log^{2} n},
 \label{eq:def_nu_n_k}
 \end{equation}
 where $\nu_{n}^{(k)}$ is a bounded sequence, repeating the same procedure that uses the first three terms of the asymptotic expansion~\eqref{eq:asympt_ber_sec}, we compute another term in the expansion of $\epsilon_{n}^{(k)}$,
 for $n\to\infty$. More precisely, it follows from~\eqref{eq:asympt_ber_sec} that
 \[
 c_{0}\!\left(k-\epsilon_{n}^{(k)}\right)\log^{2} n+c_{1}\!\left(k-\epsilon_{n}^{(k)}\right)\log n +
 c_{2}\!\left(k-\epsilon_{n}^{(k)}\right)+O\left(\frac{1}{\log n}\right)=0, \quad \mbox{ as } n\to\infty.
 \]
 By using~\eqref{eq:c_j_exp_inproof} and~\eqref{eq:def_nu_n_k}, the above equation implies that
 \[
  X_{1}^{(k)}\nu_{n}^{(k)}+X_{2}^{(k)}+O\left(\frac{1}{\log n}\right)=0, \quad \mbox{ as } n\to\infty.
 \]
 Since
 \[
  X_{1}^{(k)}=(-1)^{k-1}(k-1)! \quad\mbox{ and }\quad X_{2}^{(k)}=(-1)^{k}(k-1)!\psi(k),
 \]
 as one computes from~\eqref{eq:def_X_m_k} by using the special values~\eqref{eq:stirling_spec_val}, \eqref{eq:recip_gamma_spec_val}, and the identity $\psi(1)=-\gamma$, see~\cite[Eq.~(5.4.12)]{dlmf}, we conclude that
 \[
  \nu_{n}^{(k)}=\psi(k)+O\left(\frac{1}{\log n}\right)\!, \quad \mbox{ as } n\to\infty.
 \]
 The above formula used in~\eqref{eq:def_nu_n_k} imply the asymptotic expansion of~$x_{k}^{(n)}$ from the statement.
\end{proof}

\begin{rem}
 Note that for $k=1$, the coefficients~\eqref{eq:def_X_m_k} simplify. Namely, $X_{0}^{(1)}=0$ and $X_{m}^{(1)}=\pi_{m-1}$, for $m\in\N$. Without going into details, we write down several other terms in the asymptotic expansion of the smallest zero:
 \begin{align*}
  x_{1}^{(n)}&=1-\frac{1}{\log n}+\frac{\gamma}{\log^{2}n}-\frac{\gamma^{2}-\pi^{2}/6}{\log^{3}n}
  +\frac{\gamma^{3}-\gamma\pi^{2}/2+3\zeta(3)}{\log^{4}n}\\
  &-\frac{\gamma^{4}-\gamma^{2}\pi^{2}+12\gamma\zeta(3)-\pi^{4}/90}{\log^{5}n}
   +O\left(\frac{1}{\log^{6}n}\right)\!, \quad \mbox{ as } n\to\infty.
 \end{align*}
\end{rem}

\subsection{Another integral representation and more precise localization of zeros}

Further asymptotic analysis as well as an improvement of the localization~\eqref{eq:loc_zeros_first} will rely on an integral representation for the polynomials~$B_{n}^{(n)}$ derived below. This integral representation appeared recently in the article~\cite{blagouchine_i18} of Blagouchine, see~\cite[Eq.~(57)]{blagouchine_i18} and Appendix therein. We provide an alternative proof. We remark that the formula is a generalization of Schr\"{o}der's integral representation for the Bernoulli numbers of the second kind, see~\cite{blagouchine_jis17,schroder_zmp80}.
All fractional powers appearing below have their principal parts.

\begin{thm}\label{thm:beta_integr_repre}
 One has
 \begin{equation}
 B_{n}^{(n)}(z)=(-1)^{n}\frac{n!}{\pi}\int_{0}^{\infty}\frac{u^{z-1}}{(1+u)^{n}}\frac{\pi\cos\pi z-\log(u)\sin\pi z}{\pi^{2}+\log^{2}u}\dd u,
 \label{eq:integr_repre_ber_sec}
 \end{equation}
 for $n\in\N$ and $0<\Re z<n$. 
\end{thm}

\begin{proof}
 We again start with the integral formula~\eqref{eq:int_repre_simple} expressing the integrand as a ratio of the Gamma functions:
 \[
  B_{n}^{(n)}(z)=\int_{0}^{1}\prod_{j=1}^{n}(s+z-j)\,\dd s=\int_{0}^{1}\frac{\Gamma(s+z)}{\Gamma(s+z-n)}\dd s.
 \]
 By using the well-known identity
 \[
  \Gamma(z)\Gamma(1-z)=\frac{\pi}{\sin\pi z},
 \]
 we may rewrite the integral formula for $B_{n}^{(n)}(z)$ as
 \begin{equation}
 B_{n}^{(n)}(z)=\frac{1}{\pi}\int_{0}^{1}\sin\left(\pi(s+z) -\pi n\right)\Gamma\left(s+z\right)\Gamma\left(1-s-z+n\right)\dd s.
 \label{eq:tow_int_repre_inproof0}
 \end{equation}
 Recall that for $u,v\in\C$, $\Re u>0$, $\Re v>0$, it holds
 \[
  \frac{\Gamma(u)\Gamma(v)}{\Gamma(u+v)}=2\int_{0}^{\pi/2}\sin^{2u-1}\theta\cos^{2v-1}\theta\dd\theta,
 \]
 see \cite[Eqs.~5.12.2 and~5.12.1]{dlmf}. Using the above formula in~\eqref{eq:tow_int_repre_inproof0}, we get
 \begin{equation}
  B_{n}^{(n)}(z)=(-1)^{n}\frac{2n!}{\pi}\int_{0}^{1}\sin\left(\pi(s+z)\right)\int_{0}^{\pi/2}\sin^{2s+2z-1}\theta\cos^{2n-2s-2z+1}\theta\dd\theta\dd s,
 \label{eq:tow_int_repre_inproof1}
 \end{equation}
 where $z$ has to be restricted such that $0<\Re z<n$ to guarantee that both arguments of the Gamma functions in~\eqref{eq:tow_int_repre_inproof0} are of positive real part.
 By using Fubini's theorem, we may change the order of integration in~\eqref{eq:tow_int_repre_inproof1}. Doing also some elementary manipulations with the trigonometric functions, we arrive at the expression
 \begin{equation}
  B_{n}^{(n)}(z)=(-1)^{n}\frac{2n!}{\pi}\int_{0}^{\pi/2}\cos^{2n}\theta\tan^{2z}\theta\int_{0}^{1}\sin\left(\pi(s+z)\right)\tan^{2s-1}\theta\dd s \dd\theta,
 \label{eq:tow_int_repre_inproof2}
 \end{equation}
 for $0<\Re z<n$.

 As the last step, we evaluate the inner integral in~\eqref{eq:tow_int_repre_inproof2}. We can make use of the elementary integral
 \[
  \int e^{ax}\sin bx\,\dd x=\frac{e^{ax}}{a^{2}+b^{2}}\left(a\sin bx- b\cos bx\right)\!,
 \]
 to compute that
 \begin{equation}
   \int_{0}^{1}\sin\left(\pi(s+z)\right)\tan^{2s-1}\theta\dd s=\frac{\pi\cos\pi z-2\log(\tan\theta)\sin\pi z}{\sin\theta\cos\theta\left(\pi^{2}+4\log^{2}\tan\theta\right)},
   \label{eq:inner_int_comp_inproof}
 \end{equation}
 for $\theta\in(0,\pi/2)$. By using~\eqref{eq:inner_int_comp_inproof} in~\eqref{eq:tow_int_repre_inproof2}, we arrive at the integral representation
 \begin{equation}
    B_{n}^{(n)}(z)=2(-1)^{n}\frac{n!}{\pi}\int_{0}^{\pi/2}\cos^{2n-2}\theta\tan^{2z-1}\theta\frac{\pi\cos \pi z-2\log(\tan\theta)\sin \pi z}{\pi^{2}+4\log^{2}\tan\theta}\dd\theta,
    \label{eq:integr_repre_ber_sec_trigon}
 \end{equation}
 for $n\in\N$ and $0<\Re z<n$.  Substituting for $\tan^{2}\theta=u$ in~\eqref{eq:integr_repre_ber_sec_trigon}, we arrive at the formula~\eqref{eq:integr_repre_ber_sec}. 
\end{proof}

As a first application of the integral formula~\eqref{eq:integr_repre_ber_sec}, we improve the localization of the zeros of~$B_{n}^{(n)}$ given by the inequalities~\eqref{eq:loc_zeros_first}. It turns out that the zeros are located in a half of the respective intervals between two integers.

To do so, we rewrite the integral in~\eqref{eq:integr_repre_ber_sec} to a slightly different form. By writing the integral in~\eqref{eq:integr_repre_ber_sec} as the sum of two integrals integrating from $0$ to $1$ and from $1$ to $\infty$, respectively, substituting $u=1/\tilde{u}$ in the second one (and omitting the tilde notation afterwards),  one obtains the formula
\begin{equation}
  B_{n}^{(n)}(z)=(-1)^{n}\frac{n!}{\pi}\int_{0}^{1}\frac{1}{(1+u)^{n}}\left[\rho_{z}(u)+u^{n}\rho_{-z}(u)\right]\dd u,
  \label{eq:integr_repre_ber_sec_reform1}
\end{equation}
for $n\in\N$ and $0<\Re z<n$, where
\[
 \rho_{z}(u):=u^{z-1}\frac{\pi\cos\pi z-\log(u)\sin\pi z}{\pi^{2}+\log^{2}u}.
\]
Below, $\lfloor x\rfloor$ denotes the integer part of $x\in\R$.

\begin{thm}
 For $n\in\N$, one has
 \[
  k-\frac{1}{2}<x_{k}^{(n)}<k, \quad \mbox{ if}\quad 1\leq k \leq \bigg\lfloor \frac{n}{2}\bigg\rfloor,
 \]
 and
 \[
  k<x_{k+1}^{(n)}<k+\frac{1}{2},  \quad \mbox{ if}\quad \bigg\lfloor \frac{n+1}{2}\bigg\rfloor\leq k \leq n-1.
 \]
 Recall that $x_{n}^{(2n-1)}=n-1/2$.
\end{thm}

\begin{rem}
 The global localization of the zeros in the intervals of fixed lengths given above is the best possible. Indeed, as it is shown below, the zeros of $B_{n}^{(n)}$ located around $n/2$ cluster at half-integers as $n\to\infty$, while
 the zeros located in a left neighborhood of the point~$n$ (or in a right neighborhood of~$0$) cluster at integers as $n\to\infty$.
\end{rem}

\begin{proof}
 Clearly, if the first set of inequalities is established then the second one follows readily from the symmetry~\eqref{eq:symmetry_zeros}.
 
 For $1\leq k \leq \lfloor n/2\rfloor$, write $z=k-1/2$ in~\eqref{eq:integr_repre_ber_sec_reform1}. Since
 \[
  \rho_{k-1/2}(u)=-u^{2k-1}\rho_{-k+1/2}(u)=(-1)^{k}u^{k-3/2}\frac{\log u}{\pi^{2}+\log^{2}u},
 \]
 we have
 \[
  B_{n}^{(n)}\left(k-\frac{1}{2}\right)=(-1)^{n+k}\frac{n!}{\pi}\int_{0}^{1}\frac{u^{k-3/2}\left(1-u^{n-2k+1}\right)}{(1+u)^{n}}\frac{\log u}{\pi^{2}+\log^{2}u}\dd u.
 \]
 The integrand is obviously a negative function on $(0,1)$ for any $1\leq k \leq \lfloor n/2\rfloor$. Hence
 \[
  (-1)^{n+k+1}B_{n}^{(n)}\left(k-\frac{1}{2}\right)>0, \quad \mbox{ for}\quad 1\leq k \leq \bigg\lfloor \frac{n}{2}\bigg\rfloor.
 \]
 Taking also~\eqref{eq:int_val_sign} into account, we observe that the values 
 $B_{n}^{(n)}(k)$ and $B_{n}^{(n)}(k-1/2)$ differ in sign for $1\leq k \leq \lfloor n/2\rfloor$. Hence
 \[
  x_{k}^{(n)}\in\left(k-\frac{1}{2},k\right)\!, \quad \mbox{ for}\quad 1\leq k \leq \bigg\lfloor \frac{n}{2}\bigg\rfloor.
 \]
\end{proof}

\subsection{Asymptotic expansion of $B_{n}^{(n)}\left(z+\alpha n\right)$ and consequences for zeros}

Theorem~\ref{thm:asympt_zer_small} and Remark~\ref{rem:asympt_zer_large} give an information on the asymptotic behavior of the small or large zeros of the Bernoulli polynomials of the second kind. Bearing in mind the symmetry~\eqref{eq:symmetry_ber_sec},  the zeros of $B_{n}^{(n)}$ located ``in the middle'', i.e., around the point~$n/2$ are of interest, too. More generally, we will study the asymptotic behavior of zeros of $B_{n}^{(n)}$ that are traced along the positive real line at a speed $\alpha\in(0,1)$ by investigating the asymptotic behavior of $B_{n}^{(n)}(z+\alpha n)$, for $n\to\infty$, focusing particularly on the case $\alpha=1/2$.

In order to derive the asymptotic behavior $B_{n}^{(n)}(z+\alpha n)$ for $n$ large, we apply Laplace's method to the integral representation obtained in Theorem~\ref{thm:beta_integr_repre}. We refer reader to~\cite[Sec.~3.7]{olver_97} for a general description of Laplace's method. Here we use a particular case of Laplace's method adjusted to the situation which appears below. In particular, we need the variant where the extreme point is an inner point of the integration interval which is an easy modification of the standard form of Laplace's method where the extreme point is assumed to be one of the endpoints of the integration interval.

\begin{lem}[Laplace's method with an interior extreme point]\label{lem:laplace}
 Let $f$ be real-valued and $g$ complex-valued continuous functions on~$(0,\infty)$ independent of~$n$. Assume further that $f$ and $g$ are analytic functions at a point $a\in(0,\infty)$ where $f$ has a unique global minimum in $(0,\infty)$. Let $f_{k}$ and $g_{k}$ be the coefficients from the Taylor series
 \begin{equation}
  f(u)=f(a)+\sum_{k=0}^{\infty}f_{k}(u-a)^{k+2} \quad\mbox{ and }\quad g(u)=\sum_{k=0}^{\infty}g_{k}(u-a)^{k+\ell},
  \label{eq:tayl_ser_laplace}
 \end{equation}
 where $\ell\in\{0,1\}$ and $g_{0}\neq0$. Suppose moreover that $f_{0}\neq0$. Then, for $n\to\infty$, one has
  \begin{equation}
   \int_{0}^{\infty}e^{-nf(u)}g(u)\dd u=\frac{2\sqrt{\pi}}{\ell+1}e^{-nf(a)}\left[\frac{c_{\ell}}{n^{\ell+1/2}}+O\left(\frac{1}{n^{\ell+3/2}}\right)\!\right]
   \label{eq:int_asympt_laplace}
  \end{equation}
  provided that the integral converges absolutely for all $n$ sufficiently large. The coefficients $c_{\ell}$ are expressible in terms of $f_{\ell}$ and $g_{\ell}$ as follows:
  \[
   c_{0}=\frac{g_{0}}{2f_{0}^{1/2}} \quad\mbox{ and }\quad c_{1}=\frac{2f_{0}g_{1}-3f_{1}g_{0}}{4f_{0}^{5/2}}.
  \]
 \end{lem}

\begin{rem}\label{rem:uniform_laplace}
 Suppose that the coefficients $g_{k}=g_{k}(\xi)$ in~\eqref{eq:tayl_ser_laplace} depend continuously on an additional parameter $\xi\in K$ where $K$ is a compact subset of $\C$ and the power series for $g$ in~\eqref{eq:tayl_ser_laplace} converges uniformly in $\xi\in K$. Then the asymptotic expansion~\eqref{eq:int_asympt_laplace} holds uniformly in $\xi\in K$ as well provided that the integral converges uniformly in $\xi\in K$ for all $n$ sufficiently large.
\end{rem}

Now, we are ready to deduce an asymptotic expansion of $B_{n}^{(n)}(z+\alpha n)$ for $n\to\infty$.

\begin{thm}\label{thm:asympt_ber_alpha}
 For $\alpha\in(0,1)$ fixed, the asymptotic expansion
 \begin{align}
  \frac{(-1)^{n}\sqrt{n}}{n!\,\alpha^{\alpha n}(1-\alpha)^{(1-\alpha)n}}B_{n}^{(n)}(z+\alpha n)=&\sqrt{\frac{2}{\pi}}\frac{\alpha^{z-1/2}(1-\alpha)^{-z-1/2}}{\pi^{2}+\tau_{\alpha}^{2}}\nonumber\\
  &\times\left[\pi\cos(\pi z+\pi\alpha n)-\tau_{\alpha}\sin(\pi z+\pi\alpha n)\right]+O\left(\frac{1}{n}\right)
  \label{eq:asympt_ber_alpha}
 \end{align}
 holds locally uniformly in $z\in\C$ as $n\to\infty$, where
 \begin{equation}
  \tau_{\alpha}:=\log\frac{\alpha}{1-\alpha}.
 \label{eq:def_tau}
 \end{equation}
\end{thm}

\begin{proof}
 By writing $z+\alpha n$ instead of $z$ in~\eqref{eq:integr_repre_ber_sec}, one obtains
 \begin{equation}
  (-1)^{n}\frac{\pi}{n!} B_{n}^{(n)}(z+\alpha n)=I_{1}(n)\cos(\pi z+\pi\alpha n)-I_{2}(n)\sin(\pi z+\pi\alpha n),
 \label{eq:B_n_lin_comb_trig}
 \end{equation}
 for $-\alpha n<\Re z<(1-\alpha)n$, where
 \begin{equation}
  I_{i}(n):=\int_{0}^{\infty}e^{-nf(u)}g_{i}(u)\dd u, \quad i\in\{1,2\},
  \label{eq:def_I_12}
 \end{equation}
 and
 \begin{equation}
  f(u):=\log(1+u)-\alpha\log u,
  \label{eq:def_f}
 \end{equation}
 \vskip2pt
 \begin{equation}
  g_{1}(u):=\frac{\pi u^{z-1}}{\pi^{2}+\log^{2} u}\quad \mbox{ and } \quad g_{2}(u):=\frac{u^{z-1}\log u}{\pi^{2}+\log^{2}u}.
  \label{eq:def_g_12}
 \end{equation}
 The integrals~\eqref{eq:def_I_12} are in the suitable form for the application of Laplace's method.
 
 One easily verifies that the function $f$ defined by~\eqref{eq:def_f} has the simple global minimum at the point
 \[
  u_{\alpha}:=\frac{\alpha}{1-\alpha}.
 \]
 Further, the functions $f$ and $g_{1}$, $g_{2}$ from~\eqref{eq:def_g_12} are analytic in a neighborhood of $u_{\alpha}$ having the expansions
 \[
  f(u)=f(u_{\alpha})+\frac{(1-\alpha)^{3}}{2\alpha}(u-u_{\alpha})^{2}+O\left((u-u_{\alpha})^{3}\right)\!,
 \]
 and
 \begin{equation}
  g_{i}(u)=g_{i}(u_{\alpha})+O\left(u-u_{\alpha}\right),
 \label{eq:exp_g_12}
 \end{equation}
 for $u\to u_{\alpha}$. Moreover, the expansions for $g_{i}$ in~\eqref{eq:exp_g_12} are local uniform in $z\in\C$ as one readily checks by elementary means. 
 
 Suppose first that $\alpha\neq1/2$. Then $g_{1}(u_{\alpha})\neq0$ as well as $g_{2}(u_{\alpha})\neq0$ and Lemma~\ref{lem:laplace} applies to both $I_{1}(n)$ and $I_{2}(n)$ with $\ell=0$ resulting in the asymptotic formulas
 \begin{equation}
  I_{1}(n)=\alpha^{\alpha n}(1-\alpha)^{(1-\alpha)n}\left[\frac{\sqrt{2\pi^{3}}\alpha^{z-1/2}(1-\alpha)^{-z-1/2}}{\pi^{2}+\log^{2}\left(\alpha/(1-\alpha)\right)}\frac{1}{\sqrt{n}}+O\left(\frac{1}{n^{3/2}}\right)\right]
  \label{eq:asympt_I_1_alp}
 \end{equation}
 and
 \[
  I_{2}(n)=\alpha^{\alpha n}(1-\alpha)^{(1-\alpha)n}\left[\frac{\sqrt{2\pi}\alpha^{z-1/2}(1-\alpha)^{-z-1/2}}{\pi^{2}+\log^{2}\left(\alpha/(1-\alpha)\right)}\log\left(\frac{\alpha}{1-\alpha}\right)\frac{1}{\sqrt{n}}+O\left(\frac{1}{n^{3/2}}\right)\right]\!,
 \]
 for $n\to\infty$.
 By plugging the above expressions for $I_{i}(n)$, $i\in\{1,2\}$, into~\eqref{eq:B_n_lin_comb_trig} one gets the expansion~\eqref{eq:asympt_ber_alpha}.
 
 If $\alpha=1/2$, $g_{2}(u_{\alpha})=0$ and hence Lemma~\ref{lem:laplace} applies to $I_{2}(n)$ with $\ell=1$. It follows that $2^{n}I_{2}(n)=O(n^{-3/2})$, as $n\to\infty$, for $\alpha=1/2$. The asymptotic expansion~\eqref{eq:asympt_I_1_alp} for $I_{1}(n)$ remains unchanged even if $\alpha=1/2$. In total, taking again~\eqref{eq:B_n_lin_comb_trig} into account, we see that the expansion~\eqref{eq:asympt_ber_alpha} remains valid also for $\alpha=1/2$ since $\tau_{\alpha}$ vanishes in this case. 

 In order to conclude that the expansion~\eqref{eq:asympt_ber_alpha} is local uniform in $z\in\C$, it suffices to check that the integrals $I_{i}(n)$, $i\in\{1,2\}$, converge locally uniformly in $z\in\C$ for all $n$ sufficiently large; see Remark~\ref{rem:uniform_laplace}. Let $K\in\C$ be a compact set. Then if $z\in K$, $|\Re z|<C$ for some $C>0$. Concerning for instance $I_{1}(n)$, it holds that
 \[
  \int_{0}^{\infty}\left|\frac{u^{\alpha n+z-1}}{(1+u)^{n}}\frac{1}{\pi^{2}+\log^{2}u}\right|\dd u\leq
  \int_{0}^{1}u^{\alpha n-C-1}\dd u+\int_{1}^{\infty}\frac{u^{\alpha n+C-1}}{(1+u)^{n}}\dd u.
 \]
 For $n$ sufficiently large, the first integral on the right-hand side above can be majorized by~$1$ and the second integral converges at infinity because $\alpha<1$. Consequently, the integral $I_{1}(n)$ converges uniformly in $z\in K$ for all $n$ large enough. A similar reasoning shows that the same is true for $I_{2}(n)$ which concludes the proof.
\end{proof}

In the particular case when $\alpha=1/2$, Theorem~\ref{thm:asympt_ber_alpha} yields the following limit formulas
that can be compared with Dilcher's limit formulas for the Bernoulli polynomials of the first kind~\cite[Cor.~1]{dilcher_jat87}.

\begin{cor}
 One has
 \[
  \lim_{n\to\infty}(-1)^{n}\frac{2^{2n-1}\sqrt{n}}{(2n)!}\beta_{2n}(z)=\frac{\cos\pi z}{\pi^{3/2}}
 \]
 and
 \[
  \lim_{n\to\infty}(-1)^{n}\frac{2^{2n}\sqrt{n}}{(2n+1)!}\beta_{2n+1}(z)=\frac{\sin \pi z}{\pi^{3/2}}
 \]
 locally uniformly in $\C$, where
 \[
 \beta_{n}(z):=B_{n}^{(n)}\left(z+\frac{n}{2}\right).
 \]
\end{cor}

We can combine Theorem~\ref{thm:asympt_ber_alpha} and the Hurwitz theorem~\cite[Thm.~2.5, p.~152]{conway_78} in order to deduce the asymptotic behavior of the zeros of $B_{n}^{(n)}$ located around the point $\alpha n$ for $n$ large.

\begin{cor}\label{cor:lim_zer_alp}
 Let $\alpha\in(0,1)$. Then, for any $\ell\in\Z$, one has
 \[
  \lim_{n\to\infty}\left(x_{\lfloor\alpha n\rfloor+\ell}^{(n)}-\lfloor\alpha n\rfloor\right)=\ell-1+\frac{1}{\pi}\arccot\frac{\tau_{\alpha}}{\pi},
 \]
 where $\tau_{\alpha}$ is defined by~\eqref{eq:def_tau} and $\lfloor x\rfloor$ denotes the integer part of a real number $x$.
\end{cor}

\begin{proof}
 Replacing $z$ by $z-\alpha n+\lfloor\alpha n\rfloor$ in Theorem~\ref{thm:asympt_ber_alpha}, one deduces the limit formula
 \[
  \lim_{n\to\infty}\frac{(-1)^{n+\lfloor\alpha n\rfloor}\sqrt{n}}{n!\,\alpha^{\lfloor\alpha n\rfloor}(1-\alpha)^{n-\lfloor\alpha n\rfloor}}B_{n}^{(n)}\left(z+\lfloor\alpha n\rfloor\right)=C_{\alpha}(z)\left(\pi\cos\pi z-\tau_{\alpha}\sin\pi z\right)\!,
 \]
 where $C_{\alpha}(z)\neq0$ and the convergence is local uniform in $z\in\C$. By the Hurwitz theorem, the zeros of the polynomial
 \begin{equation}
  z\mapsto B_{n}^{(n)}\left(z+\lfloor\alpha n\rfloor\right)
 \label{eq:ber_subseq_alpha}
 \end{equation}
 cluster at the zeros of the function 
 \[
  z\mapsto \pi\cos\pi z-\tau_{\alpha}\sin\pi z
 \]
 which coincide with the solutions of the secular equation
 \begin{equation}
  \cot\pi z=\frac{\tau_{\alpha}}{\pi}.
  \label{eq:secul_eq}
 \end{equation}
 Further, it follows from~\eqref{eq:loc_zeros_first} that the zeros $x_{\lfloor\alpha n \rfloor+\ell}^{(n)}-\lfloor\alpha n\rfloor$ of~\eqref{eq:ber_subseq_alpha} satisfy
 \[
 \ell-1<x_{\lfloor\alpha n \rfloor+\ell}^{(n)}-\lfloor\alpha n\rfloor<\ell.
 \]
 Consequently, one has
 \[
 \lim_{n\to\infty}\left(x_{\lfloor\alpha n\rfloor+\ell}^{(n)}-\lfloor\alpha n\rfloor\right)=\zeta_{\ell},
 \]
 where $\zeta_{\ell}$ is the unique solution of~\eqref{eq:secul_eq} such that $\zeta_{\ell}\in(\ell-1,\ell)$.  Finally, it suffices to note that $\zeta_{\ell}=\ell-1+\zeta$, where $\zeta$ fulfills
 \[
  \cot\pi\zeta=\frac{\tau_{\alpha}}{\pi} \quad \mbox{ and }\quad \zeta\in(0,1).
 \]
\end{proof}

\begin{example}
 If we put $\alpha=1/3$, then $\tau_{1/3}=-\log 2$. Passing to the subsequences $n_{k}=3k,3k+1,3k+2$, respectively, in Corollary~\ref{cor:lim_zer_alp}, one obtains
 \begin{align*}
  \lim_{k\to\infty}\left(x_{k+\ell}^{(3k)}-k\right)=\lim_{k\to\infty}\left(x_{k+\ell}^{(3k+1)}-k\right)
  =\lim_{k\to\infty}\left(x_{k+\ell}^{(3k+2)}-k\right)&=\ell-1+\frac{1}{\pi}\arccot\left(-\frac{\log 2}{\pi}\right)\\
  &\approx\ell-1.430877,
 \end{align*}
 for any $\ell\in\Z$.
\end{example}

By taking $\alpha=1/2$ and either $n_{k}=2k$ or $n_{k}=2k+1$ in Corollary~\ref{cor:lim_zer_alp}, we get
\[
 \lim_{k\to\infty}\left(x_{k+\ell}^{(2k)}-k\right)=\lim_{k\to\infty}\left(x_{k+\ell}^{(2k+1)}-k\right)
 =\ell-\frac{1}{2},
\]
for $\ell\in\Z$ fixed. Thus, in contrast to the small or large zeros of $B_{n}^{(n)}$ that cluster at integers as shown in Theorem~\ref{thm:asympt_zer_small} and Remark~\ref{rem:asympt_zer_large}, the zeros of $B_{n}^{(n)}$ around the middle point $n/2$ cluster at half-integers as $n\to\infty$. Our next goal is to deduce more precise asymptotic expansions for the middle zeros of $B_{n}^{(n)}$.
To do so, we need to investigate the asymptotic behavior of $B_{n}^{(n)}(z+n/2)$, for $n\to\infty$, more closely.
A complete asymptotic expansion will be obtained by using the classical form of the Laplace method, see~\cite[Sec.~3.7]{olver_97}, applied together with Perron's formula for the expansion coefficients~\cite[p.~103]{wong01} adjusted slightly to our needs.

\begin{lem}[Laplace's method and Perron's formula]\label{lem:laplace_perron}
 Let $f$ be real-valued and $g$ complex-valued continous funtions on~$(0,\infty)$ independent of~$n$. Assume further that $f$ and $g$ are analytic functions at the origin where $f$ has a unique global minimum in $(0,\infty)$. Let the Maclaurin expansions of $f$ and $g$ are of the form
  \[
  f(u)=\sum_{k=0}^{\infty}f_{k}u^{k+2} \quad\mbox{ and }\quad g(u)=\sum_{k=0}^{\infty}g_{k}u^{k+\ell},
  \]
 where $\ell\in\N$ and $f_{0}\neq0$ as well as $g_{0}\neq0$. Then, for $n\to\infty$, one has
  \[
   \int_{0}^{\infty}e^{-nf(u)}g(u)\dd u\sim\sum_{k=0}^{\infty}\Gamma\left(\frac{k+\ell+1}{2}\right)\frac{c_{k}}{n^{(k+\ell+1)/2}}
  \]
  provided that the integral converges absolutely for all $n$ sufficiently large. Perron's formula for the coefficients $c_{k}$ yields
  \[
  c_{k}=\frac{1}{2k!}\frac{\dd^{k}}{\dd u^{k}}\bigg|_{u=0}\frac{g(u)u^{k+1}}{\left(f(u)\right)^{(k+\ell+1)/2}}, \quad k\in\N_{0}.
  \]
\end{lem}

\begin{thm}\label{thm:asympt_ber_alpha1/2}
 For $n\to\infty$, the complete asymptotic expansion 
 \begin{align}
   \frac{(-2)^{n}\sqrt{n}}{n!} B_{n}^{(n)}\left(z+\frac{n}{2}\right)&\sim
   \pi^{1/2}\cos\left(\pi z+\frac{\pi n}{2}\right)\sum_{k=0}^{\infty}\frac{p_{k}(z)}{2^{2k}k!}\frac{1}{n^{k}}\nonumber\\
   &-\pi^{-1/2}\sin\left(\pi z+\frac{\pi n}{2}\right)\sum_{k=0}^{\infty}\frac{q_{k}(z)}{2^{2k+1}k!}\frac{1}{n^{k+1}}
   \label{eq:asympt_ber_alp_half}
 \end{align}
 holds locally uniformly in $z\in\C$. The coefficients $p_{k}$ and $q_{k}$ are polynomials given by the formulas
 \[
  p_{k}(z)=\sum_{j=0}^{k}\binom{2k}{2j}\omega_{j}^{(k)}z^{2k-2j} \quad\mbox{ and } \quad
  q_{k}(z)=\sum_{j=0}^{k}\binom{2k+1}{2j}\omega_{j}^{(k+1)}z^{2k+1-2j},
 \]
 where
 \[
  \omega_{j}^{(k)}=\frac{\dd^{2j}}{\dd x^{2j}}\bigg|_{x=0}\,\frac{x^{2k+1}}{(\pi^{2}+x^{2})\log^{k+1/2}\cosh(x/2)}.
 \]
\end{thm}
\begin{rem}
 The first several coefficients $p_{k}$ and $q_{k}$ read
 \begin{align*}
  p_{0}(z)&=\frac{2\sqrt{2}}{\pi^{2}}, \quad p_{1}(z)=\frac{2\sqrt{2}}{\pi^{4}}\left(8\pi^{2}z^{2}+\pi^{2}-16\right)\!,\\
  p_{2}(z)&=\frac{2\sqrt{2}}{\pi^{6}}\left(64\pi^{4}z^{4}+16(5\pi^{2}-48)\pi^{2}z^{2}+\pi^{4}-160\pi^{2}+1536\right)\!,
 \end{align*}
 and
 \begin{align*}
  q_{0}(z)&=\frac{16\sqrt{2}z}{\pi^{2}}, \quad q_{1}(z)=\frac{16\sqrt{2}z}{\pi^{4}}\left(8\pi^{2}z^{2}+5\pi^{2}-48\right)\!,\\
  q_{2}(z)&=\frac{16\sqrt{2}z}{3\pi^{6}}\left(192\pi^{4}z^{4}+80(7\pi^{2}-48)\pi^{2}z^{2}+91\pi^{4}-3360\pi^{2}+23040 \right)\!.
 \end{align*}
\end{rem}

\begin{proof}
 The starting point is the equation~\eqref{eq:B_n_lin_comb_trig} with $\alpha=1/2$:
  \begin{equation}
  (-1)^{n}\frac{\pi}{n!} B_{n}^{(n)}\left(z+\frac{n}{2}\right)=I_{1}(n)\cos\left(\pi z+\frac{\pi n}{2}\right)-I_{2}(n)\sin\left(\pi z+\frac{\pi n}{2}\right)\!,
 \label{eq:B_n_lin_comb_trig_alpha_half}
 \end{equation}
 which holds true if $|\Re z|<n/2$ and where
 \[
  I_{1}(n)=\int_{0}^{\infty}\left(\frac{u^{1/2}}{1+u}\right)^{\!n}\frac{\pi u^{z-1}\,\dd u}{\pi^{2}+\log^{2}u}
  \quad\mbox{ and }\quad I_{2}(n)=\int_{0}^{\infty}\left(\frac{u^{1/2}}{1+u}\right)^{\!n}\frac{u^{z-1}\log u\,\dd u}{\pi^{2}+\log^{2}u}.
 \]
 
 First we split the integral $I_{1}(n)$ into two integrals integrating from $0$ to $1$ in the first one and from $1$ to $\infty$ in the second one. Next, we substitute for $u=e^{-x}$ in the first integral and $u=e^{x}$ in the second one. This results in the formula
 \begin{equation}
  I_{1}(n)=\frac{\pi}{2^{n-1}}\int_{0}^{\infty}\frac{\cosh(xz)}{\pi^{2}+x^{2}}\frac{\dd x}{\cosh^{n}\!\left(x/2\right)}.
 \label{eq:int_I_1_alp_half}
 \end{equation}
 Similarly one shows that
 \begin{equation}
  I_{2}(n)=\frac{1}{2^{n-1}}\int_{0}^{\infty}\frac{x\sinh(xz)}{\pi^{2}+x^{2}}\frac{\dd x}{\cosh^{n}\!\left(x/2\right)}.
  \label{eq:int_I_2_alp_half}
 \end{equation}
 
 To the integral in~\eqref{eq:int_I_1_alp_half}, we may apply Lemma~\ref{lem:laplace_perron} with
 \[
  f(x)=\log\cosh\!\left(\frac{x}{2}\right), \quad g(x)=\frac{\cosh(xz)}{\pi^{2}+x^{2}},
 \]
 and $\ell=0$ getting the expansion
 \begin{equation}
  I_{1}(n)\sim\frac{\pi}{2^{n-1}}\sum_{k=0}^{\infty}\Gamma\left(\frac{k+1}{2}\right)\frac{c_{k}}{n^{(k+1)/2}},
  \label{eq:I_1_from_Laplace_alp_half}
 \end{equation}
 where
 \[
  c_{k}=\frac{1}{2k!}\frac{\dd^{k}}{\dd x^{k}}\bigg|_{x=0}\frac{\cosh(xz)}{\pi^{2}+x^{2}}
  \frac{x^{k+1}}{\log^{(k+1)/2}\cosh(x/2)}.
 \]
 Notice that $c_{2k-1}=0$ for $k\in\N$. Next, by using the Leibnitz rule, one gets
 \[
  c_{2k}=\frac{1}{2(2k)!}\sum_{j=0}^{k}\binom{2k}{2j}\frac{\dd^{2k-2j}}{\dd x^{2k-2j}}\bigg|_{x=0}\!\left(\cosh(xz)\right)\frac{\dd^{2j}}{\dd x^{2j}}\bigg|_{x=0}\frac{x^{2k+1}}{(\pi^{2}+x^{2})\log^{k+1/2}\cosh(x/2)}
 \]
 which yields
 \begin{equation}
  c_{2k}=\frac{1}{2(2k)!}\sum_{j=0}^{2k}\binom{2k}{2j}z^{2k-2j}\omega_{j}^{(k)}=\frac{p_{k}(z)}{2(2k)!},
 \label{eq:c_2k_eq_p_k}
 \end{equation}
 where the notation from the statement has been used.
 
 By substituting from~\eqref{eq:c_2k_eq_p_k} and \eqref{eq:I_1_from_Laplace_alp_half} in the equation~\eqref{eq:B_n_lin_comb_trig_alpha_half}, one arrives at the first asymptotic series on the right-hand side of~\eqref{eq:asympt_ber_alp_half}. In order to deduce the second expansion on the right-hand side of~\eqref{eq:asympt_ber_alp_half}, one proceeds in a similar fashion applying Lemma~\ref{lem:laplace_perron} to the integral~\eqref{eq:int_I_2_alp_half} this time with $\ell=2$. The local uniformity of the expansion can be justified using the analytic dependence of the integrands in~\eqref{eq:int_I_1_alp_half} and~\eqref{eq:int_I_2_alp_half} on~$z$.
\end{proof}

Theorem~\ref{thm:asympt_ber_alpha1/2} allows to compute coefficients in the asymptotic expansion of the zeros of $B^{(n)}_{n}$ located in a fixed distance from $n/2$, for $n\to\infty$, similarly as it was done in Theorem~\ref{thm:asympt_zer_small} for the small zeros based on the asymptotic expansion from Theorem~\ref{thm:Ber_sec_compl_asympt}. Since the proof of the statement below is completely analogous to the proof of~Theorem~\ref{thm:asympt_zer_small} with the only exception that the asymptotic formula of Theorem~\ref{thm:asympt_ber_alpha1/2} is used, it is omitted.

\begin{thm}\label{thm:asympt_middle_zer}
 For any $k\in\Z$, one has
 \[
  x_{n+k}^{(2n)}=n+k-\frac{1}{2}-\frac{2k-1}{\pi^{2}n}-\frac{(2k-1)(\pi^{2}-12)}{2\pi^{4}n^{2}}+O\left(\frac{1}{n^{3}}\right)\!, \quad n\to\infty,
 \]
 and
 \[
  x_{n+k+1}^{(2n+1)}=n+k+\frac{1}{2}-\frac{2k}{\pi^{2}n}+\frac{12k}{\pi^{4}n^{2}}+O\left(\frac{1}{n^{3}}\right)\!, \quad n\to\infty.
 \]
\end{thm}

\begin{rem}
 Note the difference between the polynomial decay in the asymptotic expansion of the middle zeros of~$B_{n}^{(n)}$ and the logarithmic decay in the asymptotic formulas for the small and large zeros (Theorem~\ref{thm:asympt_zer_small} and Remark~\ref{rem:asympt_zer_large}).
\end{rem}

\begin{rem}
 More detailed expansions for the zeros located most closely to~$n/2$ read
 \begin{align*}
  x_{n}^{(2n)}=n-\frac{1}{2}+\frac{1}{\pi^{2}n}+\frac{\pi^{2}-12}{2\pi^{4}n^{2}}
  &+\frac{3\pi^{4}-100\pi^{2}+720}{12\pi^6n^{3}}\\
  &\hskip48pt+\frac{3\pi^{6}-216\pi^{4}+3856\pi^{2}-20160}{24\pi^8n^{4}}  +O\left(\frac{1}{n^{5}}\right)
 \end{align*}
 and
 \[
  x_{n+2}^{(2n+1)}=n+\frac{3}{2}-\frac{2}{\pi^{2}n}+\frac{12}{\pi^{4}n^{2}}
  -\frac{3\pi^{4}-40\pi^{2}+720}{6\pi^{6}n^{3}}+\frac{14(3\pi^{4}-44\pi^{2}+360)}{3\pi^{8}n^{4}}
  +O\left(\frac{1}{n^{5}}\right)\!,
 \]
 for $n\to\infty$.
\end{rem}

\begin{rem}
 The numbers $D_{2n}^{(2n)}:=4^{n}B_{2n}^{(2n)}(n)$ are known as the N\"{o}rlund $D$-numbers and appear in formulas for a numerical integration, see~\cite[Chp.~8, \S~7]{norlund_24}. Theorem~\ref{thm:asympt_ber_alpha1/2} implies that 
 \[
  (-1)^{n}\frac{\sqrt{2n}}{(2n)!}D_{2n}^{(2n)}\sim\sqrt{\pi}\sum_{k=0}^{\infty}\frac{\omega_{k}^{(k)}}{2^{3k}k!}\frac{1}{n^{k}}, \quad n\to\infty.
 \]
 Explicitly, the first three terms read
 \[
  (-1)^{n}\frac{\sqrt{2n}}{(2n)!}D_{2n}^{(2n)}=\frac{2\sqrt{2}}{\pi^{3/2}} + \frac{\pi^{2}-16}{2\sqrt{2}\pi^{7/2}n} + \frac{\pi^{4}-160\pi^{2}+1536}{32\sqrt{2}\pi^{11/2}n^{2}}+O\left(\frac{1}{n^{3}}\right)\!, \quad n\to\infty.
 \]
\end{rem}

\begin{rem}
 Another special case of Theorem~\ref{thm:asympt_ber_alpha1/2} yields an asymptotic expansion for the coefficients
 \[ 
 K_{2n}:=\frac{1}{(2n)!}B_{2n}^{(2n)}\left(n-\frac{1}{2}\right)
 \]
 appearing in the Gauss--Encke formula~\cite{slavic_ubpefsmf75}. Their complete asymptotic expansion reads
 \begin{align*}
  K_{2n}&\sim\frac{(-1)^{n}}{\sqrt{2\pi}4^{n+1}}\sum_{k=0}^{\infty}\frac{q_{k}\!\left(-1/2\right)}{2^{3k}k!}\frac{1}{n^{k+3/2}}\\
  &=\frac{(-1)^{n+1}}{\sqrt{2\pi}2^{2n+3}}\sum_{k=0}^{\infty}\frac{1}{2^{5k}k!}\left(\sum_{j=0}^{k}\binom{2k+1}{2j}4^{j}\omega_{j}^{(k+1)}\right)\frac{1}{n^{k+3/2}}, \quad n\to\infty,
 \end{align*}
 which explicitly yields
 \[
 K_{2n}=\frac{(-1)^{n+1}}{2^{2n-1}\pi^{5/2}n^{3/2}}\left(1+\frac{7\pi^{2}-48}{8\pi^{2}n}+\frac{3(27\pi^{4}-480\pi^{2}+2560)}{64\pi^{4}n^{2}}+O\left(\frac{1}{n^{3}}\right)\right)\!, \quad n\to\infty.
 \]
 This is a generalization of the asymptotic approximations by Steffensen and Slavi\'{c}, see~\cite{slavic_ubpefsmf75,steffensen_saj24}.
\end{rem}

\subsection{The asymptotic behavior outside the oscilatory region}

We can scale the argument of $B_{n}^{(n)}$ by $n$ and consider the polynomials
$B_{n}^{(n)}(nz)$. Their zeros are located in the interval $(0,1)$ for all $n\in\N$ which follows from~\eqref{eq:loc_zeros_first}. Consequently, the function
$z\mapsto B_{n}^{(n)}(nz)$ oscillates in $(0,1)$ and the formula~\eqref{eq:asympt_ber_alpha} shows the asymptotic behavior in the oscillatory region
\begin{align*}
 B_{n}^{(n)}(nx)=(-1)^{n}&n!\,\sqrt{\frac{2}{\pi n}}\frac{x^{xn-1/2}(1-x)^{(1-x)n-1/2}}{\pi^{2}+\log^{2}\left(x/(1-x)\right)}\\
 &\times\left[\pi\cos(\pi xn)-\log\left(\frac{x}{1-x}\right)\sin(\pi xn)+O\left(\frac{1}{n}\right)\right]\!,
\end{align*}
for $x\in(0,1)$ fixed, as $n\to\infty$. In addition, the asymptotic behavior by the edges $x=0$ and $=1$ can be obtained from Theorem~\ref{thm:Ber_sec_compl_asympt} and the symmetry relation
\begin{equation}
 B_{n}^{(n)}(z)=(-1)^{n}B_{n}^{(n)}(n-z),
\label{eq:B_n_symm}
\end{equation}
which follows from~\eqref{eq:symmetry_id_genBP}.

To complete the picture, it remains to deduce the asymptotic behavior of $B_{n}^{(n)}(nz)$ for $z$ outside the interval $[0,1]$. The asymptotic analysis is based on the following variant of the saddle point method taken from~\cite[Thm.~7.1, Chp.~4]{olver_97}; see also Perron's method in~\cite[Sec.~II.5]{wong01}.

\begin{thm}[the saddle point method]\label{thm:saddle-point}
 Let the following assumptions hold:
 \begin{enumerate}[{\upshape i)}]
  \item Functions $f$ and $g$ are independent of $n$, single valued, and analytic in a region $M\subset\C$.
  \item The integration path~$\gamma$ is independent of $n$ and its range is located in~$M$ with a possible exception of the end-points.
  \item There is a point $\xi_{0}$ located on the path~$\gamma$ which is not an end-point and is such that $f'(\xi_{0})=0$ and $f''(\xi_{0})\neq0$ (i.e., $\xi_{0}$ is a simple saddle point of~$f$).
  \item The integral
  \[
   \int_{\gamma}g(\xi)e^{-n f(\xi)}\dd\xi
  \]
  converges absolutely for all $n$ sufficiently large.
  \item One has
  \[
   \Re\left(f(\xi)-f(\xi_{0})\right)>0,
  \]
  for all $\xi\neq\xi_{0}$ that lie on the range of~$\gamma$.
 \end{enumerate}
 Then 
 \begin{equation}
  \int_{\gamma}g(\xi)e^{-n f(\xi)}\dd\xi=g(\xi_{0})e^{-nf(\xi_{0})}\sqrt{\frac{2\pi}{nf''(\xi_{0})}}\left(1+O\left(\frac{1}{n}\right)\!\right)\!, \quad \mbox{ as } n\to\infty.
 \label{eq:asympt_saddle-point}
 \end{equation}
\end{thm}

\begin{rem}\label{rem:unif_saddle-point}
 A uniform version of the saddle point method will be used again. In our case, the function $f=f(\cdot,z)$ from Theorem~\ref{thm:saddle-point} will depend analytically on an additional variable $z\in K$ where $K$ is a compact subset of $\C$. In order to conclude that the expansion~\eqref{eq:asympt_saddle-point} holds uniformly in $z\in K$, it suffices to require the assumptions (i)-(iii) and (v) to remain valid for all $z\in K$ and the integral from (iv) to converge absolutely and uniformly in $z\in K$ for all $n$ sufficiently large. The reader is referred to~\cite{neuschel_a12} for a more general uniform version of the saddle point method and the proof.
\end{rem}

Our starting point is the contour integral representation
\begin{equation}
 B_{n}^{(n)}(z)=\frac{n!}{2\pi\ii}\oint_{\gamma}\frac{(1+\xi)^{z-1}}{\xi^{n}\log(1+\xi)}\dd\xi, \quad z\in\C,
\label{eq:Ber_cont_int_repre}
\end{equation}
which follows readily from the generating function formula~\eqref{eq:gener_func_BP_sec} and~\eqref{eq:BP_sec_rel_B}. The curve $\gamma$ can be any Jordan curve with $0$ in its interior not crossing the branch cut $(-\infty,-1]$ of the integrand, i.e., the range of $\gamma$ has to be located in $\C\setminus\left((-\infty,-1]\cup\{0\}\right)$.
The principal branches of the multi-valued functions like the logarithm are chosen if not stated otherwise.
Writing $nz$ instead of $z$ in~\eqref{eq:Ber_cont_int_repre}, the contour integral can be written in the form
\begin{equation}
 B_{n}^{(n)}(nz)=\frac{n!}{2\pi\ii}\oint_{\gamma}g(\xi)e^{-nf(\xi,z)}\dd\xi, \quad z\in\C,
\label{eq:Ber_scaled_cont_int_repre}
\end{equation}
where 
\[
 f(\xi,z)=\log\xi-z\log(1+\xi) \quad \mbox{ and } \quad g(\xi)=\frac{1}{(1+\xi)\log(1+\xi)},
\]
which is suitable for the application of the saddle point method. Without loss of generality, we may restrict $z\in\C\setminus[0,1]$ to the half-plane $\Re z\leq 1/2$ due to the symmetry~\eqref{eq:B_n_symm}.

The assumptions (i),(ii), and (iv) of Theorem~\ref{thm:saddle-point} can be readily checked. Concerning the assumption~(iii), one finds that the point
\[
 \xi_{0}:=\frac{1}{z-1}
\]
is the only solution of $\partial_{\xi}f(\xi,z)=0$ and is simple. 

The most difficult part is the justification of the assumption~(v) of~Theorem~\ref{thm:saddle-point}. The idea is to investigate the level curves in the $\xi$-plane determined by the equation
\[
 \Re f(\xi,z)=\Re f(\xi_{0},z)
\]
with $z\in\C\setminus[0,1]$ and $\Re z\leq 1/2$ being fixed. This level curve, denoted as $\Omega_{0}$, is the common boundary of the two open sets
\[
 \Omega_{\pm}:=\left\{\xi\in\C\setminus\left((-\infty,-1]\cup\{0\}\right) \mid \Re f(\xi,z)\gtrless\Re f(\xi_{0},z)\right\}.
\]
Since $\Re f(\cdot,z)$ is harmonic in $\C\setminus\left((-\infty,-1]\cup\{0\}\right)$, the curves of~$\Omega_{0}$ have no end-point in $\C\setminus\left((-\infty,-1]\cup\{0\}\right)$ and, moreover, they cannot form a loop located in $\C\setminus\left((-\infty,-1]\cup\{0\}\right)$ with its interior; see, for instance,~\cite[Lemma~24 and~26]{shapiro-stampach_ca17}. Hence the only possible loop of~$\Omega_{0}$ encircles the origin (the singularity of $f(\cdot,z)$) and the level curves of $\Omega_{0}$ either goes to~$\infty$ or end at the cut $(-\infty,-1)$. In general, $\Omega_{0}$ need not be connected but the curves of~$\Omega_{0}$ intersect at exactly one point $\xi_{0}$ since it is the only stationary point of $f(\cdot,z)$.

To fulfill the assumption~(v) of~Theorem~\ref{thm:saddle-point}, we have to show that the Jordan curve~$\gamma$ can be homotopically deformed to a Jordan curve which crosses $\xi_{0}$ and is located entirely in~$\Omega_{+}$ with the only exception of the saddle point $\xi_{0}$. The following lemma will be used to justify the assumption~(v).

\begin{lem}\label{lem:justif_v}
 Let $z\in\C\setminus[0,1]$, $\Re z\leq 1/2$, be fixed. For any $\theta\in(-\pi,\pi]$, there exists $\xi\in\C\setminus\left((-\infty,-1]\cup\{0\}\right)$ with $\arg\xi=\theta$ such that $\xi\in\Omega_{0}$.
\end{lem}

\begin{proof}
 We show that $\Omega_{0}$ has a non-empty intersection with any complex ray in $\C\setminus(-\infty,0]$. First, assume $\theta\in(-\pi,\pi)$. Since 
 \[
  \lim_{r\to0-}\Re f\left(re^{\ii\theta},z\right)=-\infty \quad \mbox{ and } \quad \lim_{r\to\infty}\Re f\left(re^{\ii\theta},z\right)=\infty,
 \]
 and $\Re f\left(re^{\ii\theta},z\right)$ is continuous in $r\in(0,\infty)$, there exists $r=r(\theta)>0$ such that $r(\theta)e^{\ii\theta}\in\Omega_{0}$.
 
 Second, we verify that $\Omega_{0}$ intersects the interval $(-1,0)$. If $\Re z>0$, the task is easy since
 \[
  \lim_{x\to-1+}\Re f(x,z)=\infty \quad \mbox{ and } \quad \lim_{x\to0-}\Re f(x,z)=-\infty,
 \]
 and $\Re f(x,z)$ is continuous in $x\in(-1,0)$. 
 
 Suppose $\Re z<0$. Then $1/(\Re z-1)\in(-1,0)$. We show that $1/(\Re z-1)\in\Omega_{+}$, i.e.,
 \begin{equation}
  \Re f\left(\frac{1}{\Re z-1},z\right)>\Re f\left(\frac{1}{z-1},z\right)\!.
 \label{eq:Ref_ineq_inproof}
 \end{equation}
 Then $\Omega_{0}\cap(-1,0)\neq\emptyset$ because $\Omega_{-}$ contains a neighborhood of $0$.
 
 To verify the inequality~\eqref{eq:Ref_ineq_inproof}, we introduce the auxiliary function 
 \[
  \chi(z):=\Re\left(f\left(\frac{1}{\Re z-1},z\right)-f\left(\frac{1}{z-1},z\right)\right)\!
 \]
 and show that $\chi(z)>0$ for $\Re z<0$. Noticing that
 \[
 \frac{\partial}{\partial\xi}\bigg|_{\xi=1/(\Re z-1)}\Re f(\xi,z)=0
 \]
 one computes that
 \[
  \frac{\partial \chi}{\partial \Re z}(z)=\log\left|\frac{z}{z-1}\right|-\log\frac{\Re z}{\Re z-1}.
 \]
 It is easy to check that the above expression is positive if $\Re z<0$ and $\Im z\neq0$. Hence, if $\Im z\neq0$, then $\chi$ is a strictly increasing function of $\Re z\in(-\infty,0)$. Taking also into account that
 \[
  \lim_{\Re z\to-\infty}\chi(z)=0,
 \]
 one infers that $\chi(z)>0$ whenever $\Re z<0$ and $\Im z\neq 0$.
 
 If $\Re z<0$ and $\Im z=0$, then $1/(\Re z -1)$ coincides with the saddle point $\xi_{0}$ and hence $\Omega_{0}$ intersects the interval $(-1,0)$, too. At last, if $\Re z=0$ and $\Im z\neq0$, it suffices to check that
 \[
  \Re f\left(\frac{1}{z-1},z\right)<0,
 \]
 since
 \[
  \lim_{x\to-1+}\Re f(x,z)=0 \quad \mbox{ and } \quad \lim_{x\to0-}\Re f(x,z)=-\infty.
 \]
 The verification of the above inequality is a matter of an elementary analysis.
 \end{proof}
 
 Now, we are at the position to deduce the asymptotic formula for $B_{n}^{(n)}(nz)$, as $n\to\infty$, in the non-oscilatory regime when $z\in\C\setminus[0,1]$.

 \begin{thm}\label{thm:asympt_ber_non-osc}
  One has
  \begin{equation}
   B_{n}^{(n)}(nz)=\frac{n!}{\sqrt{2\pi n}}\frac{(z-1)^{n}}{z\log\left(z/(z-1)\right)}\left(\frac{z}{z-1}\right)^{nz+1/2}\left(1+O\left(\frac{1}{n}\right)\!\right)\!,
  \label{eq:asympt_ber_nonosc}
  \end{equation}
  for $n\to\infty$ locally uniformly in $z\in\C$ bounded away from the interval~$[0,1]$.
 \end{thm}
 
 \begin{rem}
  Note that $z/(z-1)\in(-\infty,0)$ if and only if $z\in(0,1)$. Hence the
  leading term in the asymptotic formula~\eqref{eq:asympt_ber_nonosc} is an analytic function of $z$ in $\C\setminus[0,1]$ with a branch cut~$[0,1]$.
 \end{rem}
 
 \begin{proof}
  Assume first that $z\in\C\setminus[0,1]$ and $\Re z\leq1/2$. For the application of Theorem~\ref{thm:saddle-point} to the contour integral in~\eqref{eq:Ber_scaled_cont_int_repre}, it remains to justify the crucial assumption~(v). Lemma~\ref{lem:justif_v} implies that the level curve set~$\Omega_{0}$ includes a loop which encircles the origin, intersects the interval $(-1,0)$ and the saddle point $\xi_{0}$. Moreover, the interior of this loop 
   is a subset of $\Omega_{-}$. Indeed, no other loop can separate the interior because in such a case there would be a domain in which $f(\cdot;z)$ is analytic and $\Re f(\cdot;z)$ constant on the boundary of this domain. This would imply that $f$ is a constant by general principles.
   
  The only crossing of curves in~$\Omega_{0}$ occurs at the point~$\xi_{0}$ where exactly two curves crosses at an angle of $\pi/2$ since the saddle point~$\xi_{0}$ is simple. Two of the out-going arcs encloses into the loop around the origin. The remaining two arcs either continue to $\infty$ or end at the cut~$(-\infty,-1]$. These curves cannot cross the loop at another point different from $\xi_{0}$ since $\xi_{0}$ is the only stationary point of $f(\cdot,z)$. Thus, with the only exception of the point~$\xi_{0}$, a right neighborhood of the loop, if transversed in the counter-clockwise orientation, is a subset of~$\Omega_{+}$; see Figure~\ref{fig:assum_v}.
  
  As a result, one observes that there exists a Jordan curve with $0$ in its interior, crossing the interval $(-1,0)$ and entirely located in the set~$\Omega_{+}$ except the only point~$\xi_{0}$ which belong to the image of this curve. This is the possible choice for the descent path satisfying the assumption~(v) of~Theorem~\ref{thm:saddle-point} into which the curve~$\gamma$ in~\eqref{eq:Ber_scaled_cont_int_repre} can be homotopically deformed.
  
  The asymptotic formula~\eqref{eq:asympt_ber_nonosc} follows from the application of Theorem~\ref{thm:saddle-point} and is determined up to a sign since the branch of the square root in the asymptotic formula~\eqref{eq:asympt_saddle-point} has not been specified. The local uniformity of the expansion can be justified using the fact that the integrand of~\eqref{eq:Ber_scaled_cont_int_repre} depends analytically on~$z$, see Remark~\ref{rem:unif_saddle-point}.
  
  Using~\eqref{eq:int_val_sign} and the fact that all zeros of $B_{n}^{(n)}$ are positive, we get $(-1)^{n}B_{n}^{(n)}(x)>0$ if $x<0$. By inspection of the obtained asymptotic formula~\eqref{eq:asympt_ber_nonosc}, one shows the correct choice of the sign is plus that results in~\eqref{eq:asympt_ber_nonosc}. Finally, if $z\in\C\setminus[0,1]$ with $\Re z>1/2$, one uses the symmetry~\eqref{eq:B_n_symm} together with the already obtained asymptotic formula extending the validity of~\eqref{eq:asympt_ber_nonosc} to all $z\in\C\setminus[0,1]$.
\end{proof}

\begin{figure}[htb!]
  \includegraphics[width=0.9\textwidth]{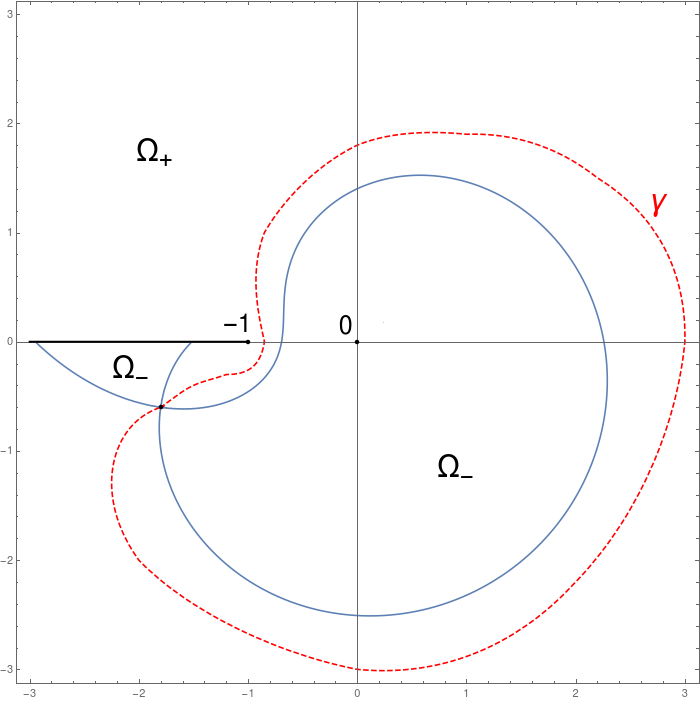}
        \caption{An illustration of the level curves $\Omega_{0}$ (blue solid line) and a possible choice of the curve $\gamma$ that fulfill the assumption~(v) of Theorem~\ref{thm:saddle-point} (red dashed line) for $z=1/2+\ii/6$.}
        \label{fig:assum_v}
\end{figure}

\begin{rem}\label{rem:asympt_zero_distr}
 It is not very surprising that the sequence of zero-counting measures of $B_{n}^{(n)}(nz)$, i.e., the uniform probability measures supported on the roots of $B_{n}^{(n)}(nz)$:
 \[
  \mu_{n}=\frac{1}{n}\sum_{k=1}^{n}\delta_{x_{k}^{(n)}/n},
 \]
 converges weakly to the uniform probability measure supported on the interval $[0,1]$. This can be verified by using the Cauchy transform of~$\mu_{n}$ and the asymptotic formula~\eqref{eq:asympt_ber_nonosc}. Indeed, for $z\in\C\setminus[0,1]$, one has
 \[
  \lim_{n\to\infty}\int_{\C}\frac{\dd\mu_{n}(\xi)}{z-\xi}=\lim_{n\to\infty}\frac{\partial}{\partial z} \log B_{n}^{(n)}(nz)=\log\frac{z}{z-1}=:C_{\mu}(z).
 \]
 Now, the Stieltjes--Perron inversion formula implies that the sequence $\{\mu_{n}\}$ converges weakly to the absolutely continuous measure supported on~$[0,1]$ whose density reads
 \[
  \frac{\dd\mu}{\dd x}(x)=\lim_{\epsilon\to0+}\frac{1}{\pi}\Im C_{\mu}(x-\ii\epsilon)=1,
 \]
for $x\in[0,1]$.
\end{rem}

\section{Final remarks and open problems}

In the end, after a short remark concerning the Euler polynomials of the second kind, we indicate several research problems related to the Bernoulli polynomials of the second kind. These interesting problems appeared during the work on this paper and remained unsolved. 

\subsection{A remark on the Euler polynomials of the second kind}

The Bernoulli polynomials are often studied jointly with the Euler polynomials since they share many similar properties~\cite{norlund_24}. The Euler polynomials of higher order are defined by the generating function
\[
 \sum_{n=0}^{\infty}E_{n}^{(a)}(x)\frac{t^{n}}{n!}=\left(\frac{2}{1+e^{t}}\right)^{a}e^{xt}.
\]
The second identity in~\eqref{eq:ber_a_ident1} remains valid at the same form even if $B_{n}^{(a)}$ is replaced by $E_{n}^{(a)}$. On the other hand, there is no simple expression for $E_{n-1}^{(n)}$ comparable with~\eqref{eq:B_n-1_n_explicit}. Hence, following the same steps as in the case of the polynomials~$B_{n}^{(n)}$ that resulted in~\eqref{eq:int_repre_simple}, one cannot deduce a simple integral expression for $E_{n}^{(n)}$. The simple integral formula~\eqref{eq:int_repre_simple} for $B_{n}^{(n)}$ was the crucial ingredience that allowed to obtain the most important results of this paper. No similar integral formula for $E_{n}^{(n)}$ is known to the best knowledge of the author. Moreover, numerical experiments indicate that the zeros of $E_{n}^{(n)}$ are not all real.

\subsection{Open problem: Alternative proofs of the reality of zeros}

Althogh the proof of the reality of zeros of~$B_{n}^{(n)}$ used in Theorem~\ref{thm:Ber_sec_zer_real} is elementary, it would be interesting to find another proof that would not be based on the particular values of $B_{n}^{(n)}$.

One way of proving the reality of zeros of a polynomial is 
based on finding a Hermitian matrix whose characteristic polynomial coincides
with the studied polynomial. This is a familiar fact, for example, for orthogonal
polynomials that are characteristic polynomials of Jacobi matrices. Since various recurrence formulas are known for the generalized Bernoulli polynomials, one may believe that there exists a Hermitian matrix $A_{n}$ with explicitly expressible elements such that $B_{n}^{(n)}(x)=\det(x-A_{n})$. We would like to stress that no such matrix was found.

Concerning this problem, one can, for instance, show the linear recursion
\[
 B_{n+1}^{(n+1)}(x)=(x-n)B_{n}^{(n)}(x)-\sum_{k=0}^{n}\binom{n}{k}\frac{b_{n-k+1}}{n-k+1}B_{k}^{(k)}(x), \quad n\in\N_{0}.
\]
by making use of~\eqref{eq:gener_func_BP_sec}, where $b_{n}:=b_{n}(0)=B_{n}^{(n)}(1)$. As a consequence, $B_{n+1}^{(n+1)}$ is the characteristic polynomials of the lower Hessenberg matrix
\[
 A_{n}=\begin{pmatrix}
        b_{1} & 1 & 0 & 0 & \dots & 0 \\
        \binom{1}{0}\frac{1}{2}b_{2} & b_{1}+1 & 1 & 0 & \dots & 0 \\
        \binom{2}{0}\frac{1}{3}b_{3} & \binom{2}{1}\frac{1}{2}b_{2} & b_{1}+2 & 1 & \dots & 0 \\
        \vdots & \vdots & \vdots & \vdots & \dots & \vdots \\
        \binom{n}{0}\frac{1}{n+1}b_{n+1} & \binom{n}{1}\frac{1}{n}b_{n} & \binom{n}{2}\frac{1}{n-1}b_{n-1} & \binom{n}{3}\frac{1}{n-2}b_{n-2} & \dots & b_{1}+n
       \end{pmatrix}\!.
\]
However, $A_{n}$ is clearly not Hermitian and the reality (and simplicity) of its eigenvalues is by no means obvious.

\subsection{Open problem: Beyond the reality of the zeros, a positivity}

A certain positivity property of a convolution-like sums with generalized Bernoulli polynomials seems to hold. This positivity would imply the reality of zeros of~$B_{n}^{(n)}$. 

It follows from the operational (umbral) calculus that~\cite[p.~94]{roman84}
\[
 B_{n}^{(a)}(x+y)=\sum_{k=0}^{n}\binom{n}{k}B_{n-k}^{(a)}(x)y^{k},
\]
for any $x,y,a\in\C$ and $n\in\N_{0}$. With the aid of the above formula, one can verify the identity
\begin{equation}
 |B_{n}^{(n)}(x+\ii y)|^{2}=|B_{n}^{(n)}(x)|^{2}+\sum_{k=1}^{\lfloor\frac{n}{2}\rfloor} \alpha_{n,k}(x)y^{2k} + \sum_{k=0}^{\lfloor\frac{n-1}{2}\rfloor}\beta_{n,k}(x) y^{2n-2k},
\label{eq:B_n_n_real_complex}
\end{equation}
where
\[
 \alpha_{n,k}(x):=(-1)^{k}\sum_{l=0}^{2k}(-1)^{l}\binom{n}{l}\binom{n}{2k-l}B_{n-l}^{(n)}(x)B_{n-2k+l}^{(n)}(x)
\]
and
\[
 \beta_{n,k}(x):=(-1)^{k}\sum_{l=0}^{2k}(-1)^{l}\binom{n}{l}\binom{n}{2k-l}B_{l}^{(n)}(x)B_{2k-l}^{(n)}(x).
\]

\begin{conjecture}
 For all $x\in\R$ and $n\in\N$, one has 
 \[
  \alpha_{n,k}(x)\geq0,\; \mbox{ for } \; 1\leq 2k\leq n,
 \;\;\mbox{ and }\;\; 
   \beta_{n,k}(x)\geq0,\; \mbox{ for } \; 0\leq 2k\leq n-1.
 \]
\end{conjecture}

If the above conjecture holds true, one would obtain from~\eqref{eq:B_n_n_real_complex}, for example, the inequality
\[
 |B_{n}^{(n)}(x+\ii y)|^{2}\geq\left(B_{n}^{(n)}(x)\right)^{\!2}+y^{2n},
\]
where we used that $\beta_{n,0}(x)=1$. From this inequality, the reality of zeros of $B_{n}^{(n)}$ immediately follows.

\subsection{Open problem: A transition between the asymptotic zero distributions of the Bernoulli polynomials of the first and second kind}

The asymptotic zero distribution~$\mu_{0}$ for Bernoulli polynomials of the first kind $B_{n}=B_{n}^{(1)}$ was found by Boyer and Goh in~\cite{boyer-goh_aam07}. It is a weak limit of the sequence of the zero-counting measures associated with the polynomials~$B_{n}$. The measure~$\mu_{0}$ is absolutely continuous, supported on certain arcs of analytic curves in $\C$, and its density is also described in~\cite{boyer-goh_aam07}. On the other hand, the asymptotic zero distribution~$\mu_{1}$ of the polynomials $B_{n}^{(n)}$ is simply the uniform probability measure supported on the interval $[0,1]$, see Remark~\ref{rem:asympt_zero_distr}.

The two measures $\mu_{0}$ and $\mu_{1}$ can be viewed as two extreme points of the asymptotic zero distribution $\mu_{\lambda}$ of the polynomials
\begin{equation}
 z\mapsto B_{n}^{(1-\lambda+\lambda n)}(nz),
\label{eq:ber_lam}
\end{equation}
where $\lambda\in[0,1]$. The order of the above generalized Bernoulli polynomial is nothing but the convex combination of $1$ and $n$. The measures $\mu_{\lambda}$ seem to be absolutely continuous and continuously dependent on~$\lambda$. An interesting research problem would be to describe the support (the zero attractor) as well as the density of $\mu_{\lambda}$ for $\lambda\in(0,1)$. For an illustration of approximate supports of $\mu_{\lambda}$ for several values of~$\lambda$, see Figure~\ref{fig:trans_lam}.

\begin{figure}[htb!]
  \includegraphics[width=0.99\textwidth]{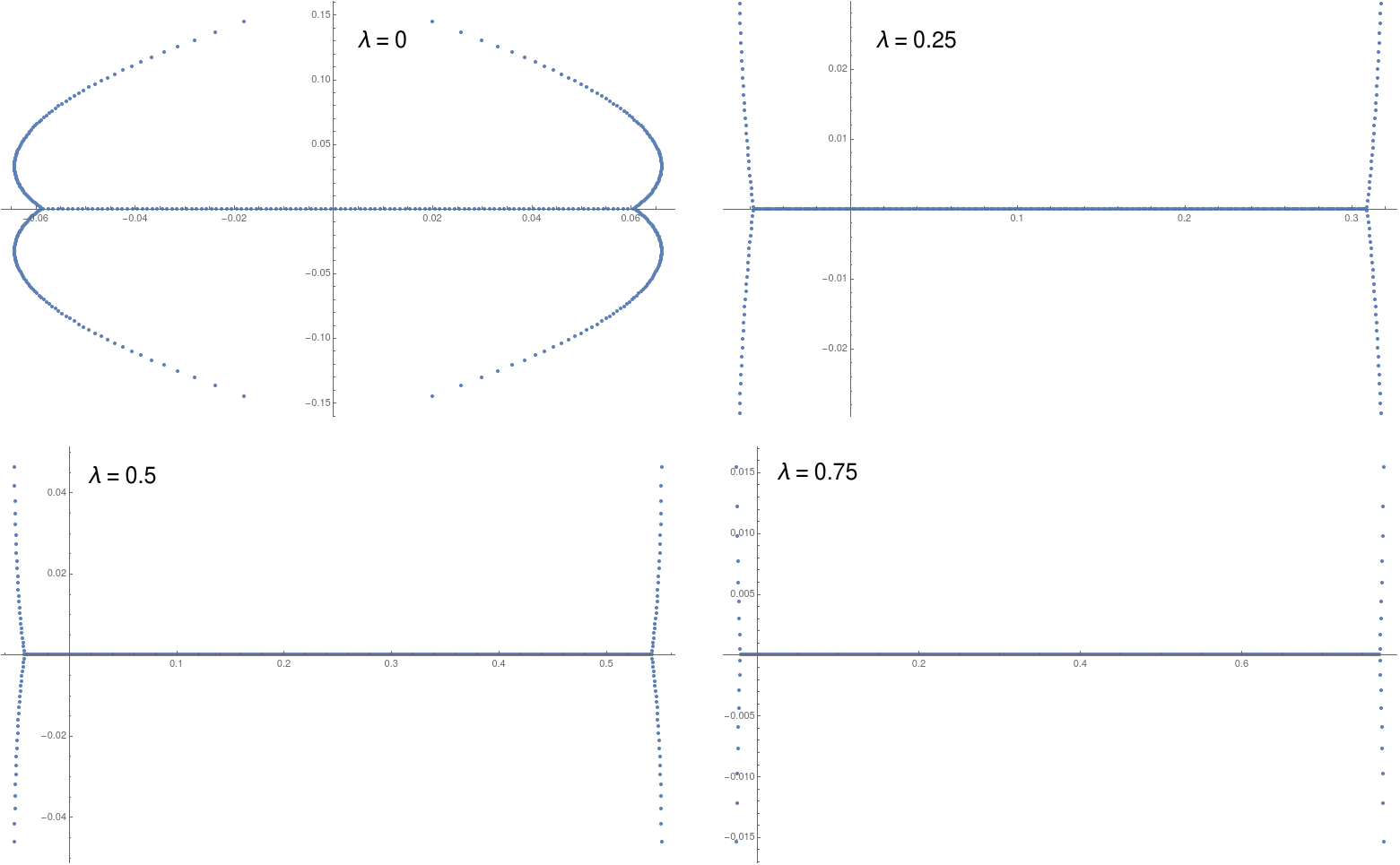}
        \caption{Plots of the zeros of the polynomials~\eqref{eq:ber_lam} in the complex plane for $n=500$ and  $\lambda\in\{0,1/4,1/2,3/4\}$ illustrating how the zero attractor of the Bernoulli polynomials (case $\lambda=0)$ deforms into the interval~$[0,1]$ (case $\lambda=1$) as $\lambda$ changes from $0$ to $1$.} 
        \label{fig:trans_lam}
\end{figure}

\subsection{Open problem: A uniform asymptotic expansion of $B_{n}^{(n)}(nx)$ for all $x\in[0,1]$}

As pointed out by an anonymous referee, it would be also an interesting open problem to find an asymptotic approximation for $B_{n}^{(n)}(nx)$ that holds uniformly for $x\in[0,1]$. Such approximation, if exists, should involve some interesting special functions reflecting the behavior of $B_{n}^{(n)}(nx)$ near the edge points $x=0$ and $x=1$.

\section*{Acknowledgement}
The author is grateful to two referees for careful reports that improved the paper.
Further, the author acknowledges financial support by the Ministry of Education, Youth and Sports of the Czech Republic 
project no. CZ.02.1.01/0.0/0.0/16\_019/0000778.

\bibliographystyle{acm}

\end{document}